\documentclass[a4paper,11 pt,reqno]{amsart} 
\usepackage{a4,amsbsy,amscd,amssymb,amstext,amsmath,mathtools,mathdots,latexsym,oldgerm,enumerate,verbatim,graphicx,ytableau,xy}
\usepackage[hang,flushmargin]{footmisc} 

\makeatletter                                       
\def\l@subsection{\@tocline{2}{0pt}{2.5pc}{2.5pc}{}}
\makeatother                                        

\makeatletter                                                  
\def\chapter{\clearpage\thispagestyle{plain}\global\@topnum\z@ 
\@afterindenttrue \secdef\@chapter\@schapter}
\makeatother

\newtheorem{thm} {Theorem} [section]
\newtheorem{prop}{Proposition} [section]
\newtheorem{lem} {Lemma} [section]

\newtheorem{cornn}{Corollary}

\theoremstyle{definition}

\newtheorem{rem} {Remark} [section]
\newtheorem{rems} [rem]{Remarks}
\newtheorem{exa} [rem] {Example}

\newtheorem{remnn}{Remark}

\newcommand{\mf}{\mathfrak}
\newcommand{\mc}{\mathcal}
\newcommand{\mb}{\mathbb}

\newcommand{\nts}{\negthinspace}     
\newcommand{\Nts}{\nts\nts}

\newcommand{\ot}{\otimes}           
\newcommand{\la}{\langle}
\newcommand{\ra}{\rangle}

\newcommand{\Hom}{{\rm Hom}}        

\newcommand{\ind}{{\rm ind}}

\newcommand{\Sym}{{\rm Sym}} 

\newcommand{\id}{{\rm id}}

\let\ttie\t
\newcommand{\tie}[1]{{\let\t\ttie \ttie#1}}
\renewcommand{\t}{\mf{t}}  

\newcommand{\GL}{{\rm GL}}

\newcommand{\SL}{{\rm SL}}

\newcommand{\Sp}{{\rm Sp}} 

\newcommand{\ve}{\varepsilon}

\makeatletter
\def\vcdots{\vbox{\baselineskip4\p@ \lineskiplimit\z@
\kern3\p@\hbox{.}\hbox{.}\hbox{.}\Nts\nts\kern3\p@}}
\makeatother

\xyoption{all}

\begin{document}

\title{A combinatorial translation principle and diagram combinatorics for the symplectic group}

\begin{abstract}
Let $k$ be an algebraically closed field of characteristic $p>2$. We compute the Weyl filtration multiplicities in indecomposable tilting modules and the decomposition numbers for the symplectic group over $k$ in terms of cap-curl diagrams under the assumption that $p$ is bigger than the greatest hook length in the largest partition involved.
As a corollary we obtain the decomposition numbers for the Brauer algebra under the same assumptions. Our work combines ideas from work of Cox and De Visscher and work of Shalile with techniques from the representation theory of reductive groups.
\end{abstract}
\author[H.\ Li]{Henry Li}
\address
{School of Mathematics,
University of Leeds,
LS2 9JT, Leeds, UK}
\email{mmzli@leeds.ac.uk}

\author[R.\ Tange]{Rudolf Tange}
\address
{School of Mathematics,
University of Leeds,
LS2 9JT, Leeds, UK}
\email{R.H.Tange@leeds.ac.uk}

\maketitle
\markright{\MakeUppercase{A combinatorial translation principle for} $\Sp_n$}

\section{Introduction}
The present paper concerns the symplectic group and the Brauer algebra. In the companion paper \cite{T} the analogous results for the general linear group and the walled Brauer algebra are obtained.

Let $\Sp_n$, $n=2m$, be the symplectic group over an algebraically closed field $k$ of characteristic $p>2$, and let $V$ be the natural module. Since Williamson \cite{Wi} disproved Lusztig's conjecture for $\SL_n$ and $p$ bigger than any linear bound in $n$, it has become more interesting to determine decomposition numbers of reductive groups for special sets of weights.
In the present paper we do this for $\Sp_n$ and dominant weights for which $p$ is bigger than the greatest hook length.

The Brauer algebra is a cellular algebra and an interesting problem is to determine its decomposition numbers. In characteristic $0$ this was first done in \cite{Mar} and in \cite{CdV} an alternative proof was given which included the analogous result for the walled Brauer algebra. In \cite{Sh} the decomposition numbers of the Brauer algebra were determined in characteristic $p>r$. All these results are in terms of certain cap (or cap-curl) diagrams.

In characteristic $0$ there is a well-known relation between certain representations of $\Sp_n$ and the representations of the Brauer algebra $B_r(-n)$, given by the double centraliser theorem for  their actions on $V^{\ot r}$. In characteristic $p$ such a connection doesn't follow from the double centraliser theorem and requires more work. This was done in \cite{DT} by means of the symplectic Schur functor.

In the present paper we determine the Weyl filtration multiplicities in the indecomposable tilting modules $T(\lambda)$ and the decomposition numbers for the induced modules $\nabla(\lambda)$ of $\Sp_n$ when $\lambda$ has greatest hook length less than $p$. Using the symplectic Schur functor we then obtain from the first multiplicities the decomposition numbers of the Brauer algebra under the assumption that $p$ is bigger than the greatest hook length in the largest partition involved. Since we use the transposed labels, our description of the decomposition numbers is considerably simpler than that of \cite{Sh}. Our main tools are the ``reduced" Jantzen Sum Formula, truncation, and refined translation functors.

Our approach is mainly based on \cite{CdV} and \cite{Sh}, but the combinatorial ideas go back, via work of Brundan-Stroppel, see e.g. \cite{BS}, to work of Boe \cite{Boe} and Lascoux-Sch\"utzenberger \cite{LS}. In the latter two papers the combinatorial structures are bracket expressions and trees, rather than cap diagrams. Our translation functors are modifications of those in \cite[II.7]{Jan}. The idea of translation functors is based on the linkage principle and goes back to category $\mc O$, see \cite[Ch~7]{Hum}.

The paper is organised as follows. In Section~\ref{s.prelim} we introduce the necessary notation and state two results about quasihereditary algebras. In Section~\ref{s.JSF} we show that certain terms in the Jantzen Sum Formula may be omitted. This leads to a ``strong linkage principle" in terms of a partial order $\preceq$, and the existence of nonzero homomorphisms between certain pairs of induced modules, see Proposition~\ref{prop.linkage}. In Section~\ref{s.translation} we prove the two basic results about translation that we will use: Propositions~\ref{prop.trans_equivalence} and ~\ref{prop.trans_projection}. For this we use refined translation functors defined on certain truncations of the category of $\Sp_n$-modules. In Section~\ref{s.arrow_diagrams} we introduce arrow diagrams to represent the weights that satisfy our condition, and we show that the nonzero terms in the reduced Jantzen Sum Formula and the pairs of weights for which we proved the existence of nonzero homomorphisms between the induced modules have a simple description in terms of arrow diagrams, see Lemma~\ref{lem.JSF-arrows}. The order $\preceq$ and conjugacy under the dot action also have a simple description in terms of the arrow diagram, see Remark~\ref{rems.preceq}.1. In Section~\ref{s.filt_mult} we prove our first main result, Theorem~\ref{thm.filt_mult}, which describes the Weyl filtration multiplicities in certain indecomposable tilting modules in terms of cap-curl diagrams. As a corollary we obtain the decomposition numbers of the Brauer algebra under the assumption that $p$ is bigger than the greatest hook length in the largest partition involved. In Section~\ref{s.dec_num} we prove our second main result, Theorem~\ref{thm.dec_num}, which describes the decomposition numbers for certain induced modules in terms of cap-curl codiagrams.

\section{Preliminaries}\label{s.prelim}
Throughout this paper $G$ is a reductive group over an algebraically closed field $k$ of characteristic $p>2$, $T$ is a maximal torus of $G$ and $B^+$ is a Borel subgroup of $G$ containing $T$.
We denote the group of weights relative to $T$, i.e. the group of characters of $T$, by $X$. For $\lambda,\mu\in X$ we write $\mu\le\lambda$ if $\lambda-\mu$ is a sum of positive roots (relative to $B^+$). The Weyl group of $G$ relative to $T$ is denoted by $W$ and the set of dominant weights relative to $B^+$ is denoted by $X^+$. In the category of (rational) $G$-modules, i.e. $k[G]$-comodules, there are several special families of modules. For $\lambda\in X^+$ we have the irreducible $L(\lambda)$ of highest weight $\lambda$, and the induced module $\nabla(\lambda)=\ind_B^Gk_\lambda$, where $B$ is the opposite Borel subgroup to $B^+$ and $k_\lambda$ is the 1-dimensional $B$-module afforded by $\lambda$. The Weyl module and indecomposable tilting module associated to $\lambda$ are denoted by $\Delta(\lambda)$ and $T(\lambda)$. To each $G$-module $M$ we can associate its formal character ${\rm ch}\,M=\sum_{\lambda\in X}\dim M_\lambda e(\lambda)\in(\mb Z X)^W$, where $M_\lambda$ is the weight space associated to $\lambda$ and $e(\lambda)$ is the basis element corresponding to $\lambda$ of the group algebra $\mb Z X$ of $X$ over $\mb Z$. Composition and good or Weyl filtration multiplicities are denoted by $[M:L(\lambda)]$ and $(M:\nabla(\lambda))$ or $(M:\Delta(\lambda))$. For a weight $\lambda$, the character $\chi(\lambda)$ is given by Weyl's character formula \cite[II.5.10]{Jan}. If $\lambda$ is dominant, then ${\rm ch}\,\nabla(\lambda)={\rm ch}\,\Delta(\lambda)=\chi(\lambda)$. The $\chi(\lambda)$, $\lambda\in X^+$, form a $\mb Z$-basis of $(\mb Z X)^W$.
For $\alpha$ a root and $l\in\mb Z$, let $s_{\alpha,l}$ be the affine reflection of $\mb R\ot_{\mb Z}X$ defined by $s_{\alpha,l}(x)=x-a\alpha$, where $a=\la x,\alpha^\vee\ra-lp$. Mostly we replace $\la-,-\ra$ by a $W$-invariant inner product and then the cocharacter group of $T$ is identified with a lattice in $\mb R\otimes_{\mb Z}X$ and $\alpha^\vee=\frac{2}{\la\alpha,\alpha\ra}\alpha$. We have $s_{-\alpha,l}=s_{\alpha,-l}$ and the affine Weyl group $W_p$ is generated by the $s_{\alpha,l}$.
We denote the half sum of the positive roots by $\rho$ and define the dot action of $W_p$ on $\mb R\ot_{\mb Z}X$ by $w\cdot x=w(\lambda+\rho)-\rho$. The lattice $X$ is stable under the dot action. The \emph{linkage principle} \cite[II.6.17,7.2]{Jan} says that if $L(\lambda)$ and $L(\mu)$ belong to the same $G$-block, then $\lambda$ and $\mu$ are $W_p$-conjugate under the dot action. We refer to \cite{Jan} part II for more details.

Unless stated otherwise, $G$ will be the symplectic group $\Sp_n$, $n=2m$, given by $\Sp_n=\{A\in\GL_n\,|\,A^TJA=J\}$, where $J=\begin{bmatrix}0&I_m\\-I_m&0\end{bmatrix}$ and $A^T$ is the transpose of $A$. Note that, since $J^2=-I_n$, $A\in\Sp_n$ implies that $A^TJ=(-AJ)^{-1}$ and therefore also $AJA^T=J$. The natural $G$-module $k^n$ is denoted by $V$. We let $T$ be the group of diagonal matrices in $\Sp_n$, i.e. the matrices $diag(d_1,\ldots,d_n)$ with $d_id_{i+m}=1$ for all $i\in\{1,\ldots,m\}$. Now $X$ is naturally identified with $\mb Z^m$ such that the $i$-th diagonal coordinate function corresponds to the $i$-th standard basis element $\ve_i$ of $\mb Z^m$. We let $B^+$ be the Borel subgroup corresponding to the set of positive roots $\ve_i\pm\ve_j$, $1\le i<j\le m$, $2\ve_i$, $1\le i\le m$.
We can now identify the dominant weights with $m$-tuples $(\lambda_1,\ldots,\lambda_m)$ with $\lambda_1\ge\lambda_2\ge\cdots\ge\lambda_m\ge0$, or with partitions $\lambda$ with $l(\lambda)\le m$, where $l(\lambda)$ denotes the length of a partition. We will also identify them with the corresponding Young diagrams. Partitions with parts $<10$ may be written in ``exponential form": $(5,5,4,3,2)$ is denoted by $(5^2432)$, where we sometimes omit the brackets. We denote the subgroup of $W_p$ generated by the $s_{\alpha,l}$, $\alpha=\ve_i\pm\ve_j$, $1\le i<j\le s$ or $\alpha=2\ve_i$, $1\le i\le s$ by $W_p(C_s)$ and we denote the subgroup of $W_p$ generated by the $s_{\alpha,l}$, $\alpha=\ve_i\pm\ve_j$, $1\le i<j\le s$ by $W_p(D_s)$. The group $W$ acts on $\mb Z^m$ by permutations and sign changes, and $W_p\cong W\ltimes pX_{ev}$, where $X_{ev}=\{\lambda\in\mb Z^m\,|\,|\lambda|\text{ even}\}$ is the type $C_m$ root lattice and $|\lambda|=\sum_{i=1}^m\lambda_i$. Note that $W_p(C_s)\cong W(C_s)\ltimes pX_{ev}(C_s)$ and $W_p(D_s)\cong W(D_s)\ltimes pX_{ev}(C_s)$, where $X_{ev}(C_s)$ consists of the vectors in $X_{ev}$ which are $0$ at the positions $>s$, $W(C_s)$ is generated by the $s_\alpha=s_{\alpha,0}$, $\alpha=\ve_i\pm\ve_j$, $1\le i<j\le s$ or $\alpha=2\ve_i$, $1\le i\le s$, and $W(D_s)$ is generated by the $s_\alpha$, $\alpha=\ve_i\pm\ve_j$, $1\le i<j\le s$. The group $W(D_s)$ acts by permutations and an even number of sign changes. We have $$\rho=(m,m-1,\ldots,1)\,.$$
It is easy to see that if $\lambda,\mu\in X$ are $W_p$-conjugate and equal at the positions $>s$, then they are $W_p(C_s)$-conjugate. The same applies for the dot action.

To obtain our results we will have to make use of quasihereditary algebras. We refer to \cite[Appendix]{D} and \cite[Ch~A]{Jan} for the general theory.
For a subset $\Lambda$ of $X^+$ and a $G$-module $M$ we say that $M$ \emph{belongs to $\Lambda$} if all composition factors have highest weight in $\Lambda$ and we denote by $O_\Lambda(M)$ the largest submodule of $M$ which belongs to $\Lambda$. For a quasihereditary algebra one can make completely analogous definitions. We denote the category of $G$-modules which belong to $\Lambda$ by $\mc C_\Lambda$.
The following result is part of the folklore.
\begin{lem}\label{lem.quasihereditary}
Let $A$ be a quasihereditary algebra with partially ordered labelling set $(\Lambda,\le)$ for the irreducibles. Let $\sqsubseteq$ be a partial order on $\Lambda$ such that $[\Delta(\lambda):L(\mu)]\ne0$ or $[\nabla(\lambda):L(\mu)]\ne0$ implies $\mu\sqsubseteq\lambda$. Then $A$ is also quasihereditary for $\sqsubseteq$ and the (co)standard modules (and therefore also the tilting modules) are the same as for $\le$.
\end{lem}
Later on we will use this result as follows. Recall that a subset $\Lambda'$ of a partially ordered set $(\Lambda,\le)$ is called \emph{saturated} if $\lambda\in\Lambda'$ and $\mu\le\lambda$ implies $\mu\in\Lambda'$ for all $\lambda,\mu\in\Lambda$.
Now we first we choose a certain finite saturated subset $\Lambda_0$ of $(X^+,\le)$ and form the quasihereditary algebra $O_{\Lambda_0}(k[G])^*$, where $k[G]$ is considered as $G$-module via the right multiplication action of $G$ on itself. Then we replace the partial ordering by a suitable ordering $\sqsubseteq$ and truncate the algebra to a smaller $\sqsubseteq$-saturated set $\Lambda$. For the resulting quasihereditary algebra $A$ the irreducible, standard/costandard and tilting modules are the irreducible, Weyl/induced and tilting modules for $G$ with the same label. So we will always have that $[\Delta(\lambda):L(\mu)]=[\nabla(\lambda):L(\mu)]$. Furthermore, we have for $\lambda\in\Lambda$ and the $G$-injective hull $I(\lambda)$ of $L(\lambda)$ that $I_\Lambda(\lambda)\stackrel{\text{def}}{=}O_\Lambda(I(\lambda))$ is the $A$-injective hull of $L(\lambda)$. Note that $\mc C_\Lambda$ is the category of $A$-modules.

We will need one more general result about quasihereditary algebras.
\begin{lem}\label{lem.quasihereditary2}
Let $A$ be a quasihereditary algebra with partially ordered labelling set $(\Lambda,\le)$ for the irreducibles. Let $\lambda\in\Lambda$ and assume $\mu\in\Lambda$ is maximal with $\mu<\lambda$. Then $\dim\Hom_A(\nabla(\lambda),\nabla(\mu))=[\nabla(\lambda):L(\mu)]$.
\end{lem}
\begin{proof}
Using \cite[Prop~A.2.2(i)]{D} and the argument in the proof of \cite[II.6.24]{Jan} we get $\Hom_A(\nabla(\lambda),\nabla(\mu))=\Hom_A(\nabla(\lambda),I(\mu))$, where $I(\mu)$ is the injective hull of $L(\mu)$. The result now follows by taking dimensions.
\end{proof}

\section{The reduced Jantzen Sum Formula}\label{s.JSF}
In this section we study the Jantzen Sum Formula for the symplectic group $\Sp_n$. We want to strengthen the results from \cite[Sect~3]{DT}.
Assume for the moment that $G$ is any reductive group. Jantzen has defined for every Weyl module $\Delta(\lambda)$ of $G$ a descending filtration $\Delta(\lambda)=\Delta(\lambda)^0\supseteq\Delta(\lambda)^1\supseteq\cdots$ such that $\Delta(\lambda)/\Delta(\lambda)^1\cong L(\lambda)$ and $\Delta(\lambda)^i=0$ for $i$ big enough. The Jantzen sum formula \cite[II.8.19]{Jan} relates the formal characters of the $\Delta(\lambda)^i$ with the Weyl characters $\chi(\mu)$, $\mu\in X^+$:
\begin{equation}\label{eq.JSF}
\sum_{i>0}{\rm ch}\,\Delta(\lambda)^i=\sum\nu_p(lp)\chi(s_{\alpha,l}\cdot\lambda)\ ,
\end{equation}
where the sum on the right is over all pairs $(\alpha,l)$, with $l$ an integer $\ge1$ and $\alpha$ a positive root such that $\la\lambda+\rho,\alpha^\vee\ra-lp>0$, and $\nu_p$ is the $p$-adic valuation. Here $\chi(\mu)=0$ if and only if $\la\mu+\rho,\alpha^\vee\ra=0$ for some $\alpha>0$, and if $\chi(\mu)\ne0$, then $\chi(\mu)=\det(w)\chi(w\cdot\mu)$, where $w\cdot\mu$ is dominant for a unique $w\in W$. See \cite[II.5.9(1)]{Jan}. We denote the RHS of \eqref{eq.JSF} by $JSF(\lambda)$.

Now return to our standard assumption $G=\Sp_n$. For $\lambda\in X$ we have that $\chi(\lambda)\ne0$ if and only if
\begin{align*}
&(\lambda+\rho)_i\ne0\text{\quad for all\ }i\in\{1,\ldots,m\}\text{\quad and}\\
&(\lambda+\rho)_i\ne\pm(\lambda+\rho)_j\text{\quad for all\ }i,j\in\{1,\ldots,m\}\text{\ with\ }i\ne j.
\end{align*}

For the remainder of this section $\lambda$ is a \emph{$p$-core}, unless stated otherwise. This means that for all $i\in\{1,\ldots,m\}$ and all integers $l\ge1$, $(\lambda+\rho)_i-lp$ must occur in $\lambda+\rho$, provided it is $>0$.\footnote{\hbox{This is equivalent to the definition in \cite[Ex~I.1.8]{Mac}, but note that we work with a different $\rho$.}}

\begin{lem}\label{lem.e_i-e_j}
Assume $\alpha=\ve_i-\ve_j$, $1\le i<j\le m$ and $\la\lambda+\rho,\alpha^\vee\ra=a+lp$, $a,l>0$. Then $\chi(s_{\alpha,l}\cdot\lambda)=0$.
\end{lem}
\begin{proof}
We have $(\lambda+\rho)_i-(\lambda+\rho)_j=\la\lambda+\rho,\alpha^\vee\ra=a+lp$, $a,l>0$.
Now $(\lambda+\rho)_i-lp=(\lambda+\rho)_j+a=s_{\alpha,l}(\lambda+\rho)_j$ must occur in $\lambda+\rho$ and it clearly can't occur in position $i$ or $j$, so $s_{\alpha,l}(\lambda+\rho)$ contains a repeat and $\chi(s_{\alpha,l}\cdot\lambda)=0$.
\end{proof}

\begin{lem}\label{lem.cancellation}
Let $\Phi_1$ be the set of roots $\alpha=\ve_i+\ve_j$, $1\le i<j\le m$ for which $(\lambda+\rho)_j-(\la\lambda+\rho,\alpha^\vee\ra-lp)<0$, and let $\Phi_2$ be the set of roots $2\ve_i$, $1\le i\le m$. Furthermore, let $S_1$ be the set of pairs $(\alpha,l)$ such that $\alpha\in\Phi_1$, $l$ an integer $\ge1$, $\la\lambda+\rho,\alpha^\vee\ra-lp>0$ and $\chi(s_{\alpha,l}\cdot\lambda)\ne0$, and let $S_2$ be the corresponding set for $\Phi_2$. Then there exists a map $\varphi:S_1\to\Phi_2$ such that:
\begin{enumerate}[{\rm(i)}]
\item $(\alpha,l)\mapsto(\varphi(\alpha,l),l)$ is a bijection from $S_1$ onto $S_2$.
\item $\chi(s_{\alpha,l}\cdot\lambda)=-\chi(s_{\varphi(\alpha,l),l}\cdot\lambda)$.
\end{enumerate}
Furthermore, if $\alpha=\ve_i+\ve_j$, $1\le i<j\le m$, and $l$ is an integer $\ge1$ such that $j>l(\lambda)$, $\la\lambda+\rho,\alpha^\vee\ra-lp>0$ and $\chi(s_{\alpha,l}\cdot\lambda)\ne0$, then $(\alpha,l)\in S_1$.
\end{lem}

\begin{proof}
Let $(\alpha,l)\in S_1$. Write $\alpha=\alpha^\vee=\ve_i+\ve_j$, $1\le i<j\le m$ and put $a=\la\lambda+\rho,\alpha^\vee\ra-lp$. We have $(\lambda+\rho)_i+(\lambda+\rho)_j=\la\lambda+\rho,\alpha^\vee\ra=a+lp$ and $s_{\alpha,l}(\lambda+\rho)=\lambda+\rho-a\alpha$. Since $(\lambda+\rho)_i-lp=a-(\lambda+\rho)_j>0$, this value must occur in a position $\ne i$ in $\lambda+\rho$. If this position were $\ne j$, then $s_{\alpha,l}(\lambda+\rho)$ would contain a repeat up to sign. So $a=2(\lambda+\rho)_j$.

Now put $\varphi(l,\alpha)=2\ve_i$. Note that $(2\ve_i)^\vee=\ve_i$. So $\la\lambda+\rho,\varphi(l,\alpha)^\vee\ra=(\lambda+\rho)_i=(\lambda+\rho)_j+lp$.
Furthermore, $s_{\alpha,l}(\lambda+\rho)$ is obtained from $s_{\varphi(l,\alpha),l}(\lambda+\rho)$ by changing the sign of the $j^{\rm th}$ coordinate. This proves (ii) and that $(\alpha,l)\mapsto(\varphi(\alpha,l),l)$ is an injection from $S_1$ to $S_2$.

Now let $(\beta,l)\in S_2$ and write $\beta=2\ve_i$. Then $(\lambda+\rho)_i-lp=\la\lambda+\rho,\beta^\vee\ra-lp>0$ and must occur in some position $j>i$ in $\lambda+\rho$. Put $\alpha=\ve_i+\ve_j$. Then $\la\lambda+\rho,\alpha^\vee\ra-lp=(\lambda+\rho)_j+(\lambda+\rho)_i-lp=2(\lambda+\rho)_j>0$, $2(\lambda+\rho)_j-(\lambda+\rho)_j>0$. Furthermore, $s_{\alpha,l}(\lambda+\rho)$ is obtained from $s_{\beta,l}(\lambda+\rho)$ by changing the sign of the $j^{\rm th}$ coordinate. So $(\alpha,l)\in S_1$ and $\varphi(\alpha,l)=\beta$. This proves (i).

We now prove the final assertion. Assume $\alpha,i,j,l$ are as stated and put $a=\la\lambda+\rho,\alpha^\vee\ra-lp$. Note that $\lambda_j=0$. If $\rho_j-a>0$, then it must occur in $\lambda+\rho$ after the $j$-th position, since $j>l(\lambda)$ and the last $m-l(\lambda)$ entries of $(\lambda+\rho)$ form an interval with smallest value $1$. So $s_{\alpha,l}(\lambda+\rho)$ would contain a repeat, which contradicts $\chi(s_{\alpha,l}\cdot\lambda)\ne0$. If $\rho_j-a=0$, then $s_{\alpha,l}(\lambda+\rho)$ contains $0$ which is also impossible. Therefore, $(\lambda+\rho)_j-a=\rho_j-a<0$.
\end{proof}

\begin{exa}
If $\lambda$ is not a $p$-core, then there may be surviving contributions coming from a root $2\ve_i$.
For example, when $p=5$, $m=2$ and $\lambda=(41)$, then $JSF(\lambda)=-\chi(1)+\chi(21)$, where the terms come from $\ve_1+\ve_2$ and $2\ve_1$, both with $l=1$. Since $JSF(21)=\chi(1)$, we get $JSF(\lambda)={\rm ch}\,L(21)$.

If $\lambda$ is a $p$-core, then a root $\ve_i+\ve_j$, $1\le i<j\le l(\lambda)$ can make a surviving contribution for one $l$ and a cancelling one for another $l$. For example, take $p=3$, $m=3$ and $\lambda=(21^2)$. Then $JSF(\lambda)=\chi(\emptyset)-\chi(\emptyset)+\chi(11)-\chi(0)$, where the terms come from the $(\alpha,l)$-pairs $(2\ve_1,1)$, $(\ve_1+\ve_2,2)$, $(\ve_1+\ve_3,2)$ and $(\ve_1+\ve_3,1)$, respectively. Here the first contribution of $\ve_1+\ve_3$ is surviving, since $(\lambda+\rho)_3-a=2-1>0$ and the second one cancelling, since $(\lambda+\rho)_3-a=2-4<0$.
\end{exa}

By the previous two lemmas we may, when $\lambda$ is a $p$-core, restrict the sum on the RHS of \eqref{eq.JSF} to pairs $(\alpha,l)$ with $\alpha=\ve_i+\ve_j$, $1\le i<j\le l(\lambda)$, $\la\lambda+\rho,\alpha^\vee\ra=a+lp$, $a,l\ge 1$ and $(\lambda+\rho)_j-a>0$ (and $\chi(s_{\alpha,l}\cdot\lambda)\ne0$). We will refer to this sum as the \emph{reduced sum} and to the whole equality as the \emph{reduced Jantzen Sum Formula}. For $\mu,\nu\in\mb Z^m$ we write $\mu\subseteq\nu$ when $\mu_i\le\nu_i$ for all $i\in\{1,\ldots,m\}$, and we denote the weakly decreasing permutation of $\mu$ by ${\rm sort}(\mu)$. The next lemma shows that, when working with Weyl characters, the nonzero terms in the reduced sum have distinct Weyl characters.

\begin{lem}\label{lem.inclusion}
Let $\alpha=\ve_i+\ve_j$, $1\le i<j\le l(\lambda)$, be a positive root with $\la\lambda+\rho,\alpha^\vee\ra=a+lp$, $a,l\ge 1$, $(\lambda+\rho)_j-a>0$ and $\chi(s_{\alpha,l}\cdot\lambda)\ne0$. Then the entries of $s_{\alpha,l}(\lambda+\rho)$ are distinct and strictly positive. Put differently, for some (or all) $t\in\{l(\lambda),\ldots,m\}$, the first $t$ entries of $s_{\alpha,l}(\lambda+\rho)$ are distinct and $>m-t$. 
Now put $\mu={\rm sort}(s_{\alpha,l}(\lambda+\rho))-\rho$. Then $\mu$ is a partition with $\mu\subsetneqq\lambda$ which is $W_p(D_{l(\lambda)})$-conjugate to $\lambda$ under the dot action. Furthermore, the map $(\alpha,l)\mapsto\mu$ is injective.
\end{lem}

\begin{proof}
Since $(\lambda+\rho)_i-a>(\lambda+\rho)_j-a>0$ it is clear that all entries of $s_{\alpha,l}(\lambda+\rho)$ are (distinct and) strictly positive.
We have $s_{\alpha,l}(\lambda+\rho)\subseteq \lambda+\rho$ and therefore $\mu\subseteq\lambda$, since $\lambda$ is weakly decreasing. All assertions, except the final one are now clear. The set of values in $s_{\alpha,l}(\lambda+\rho)$ is obtained by choosing two values in $\lambda+\rho$ and strictly lowering these to values which don't occur in any of the other $m-2$ positions in $\lambda+\rho$. If two values disappear (and two new ones are introduced), then it is obvious how to recover $i$, $j$, $a$ and $l$ from the value set of $s_{\alpha,l}(\lambda+\rho)$. If only one value $x$ disappears and one new value $y$ is introduced, then $x=(\lambda+\rho)_i$, $a=(x-y)/2$ and $(\lambda+\rho)_j=(x+y)/2$ and this gives us $i$, $j$ and $l$.
\end{proof}

\begin{exa}
In the reduced sum one $\alpha$ can contribute with more than one $l$-value. Furthermore, $\Delta(\lambda)$ may have composition factors $L(\mu)$ with $\mu\nsubseteq\lambda$.
For example, take $p=3$, $m=4$ and $\lambda=(642)$. Then $\lambda$ is a $p$-core and we have $JSF(1^2)=0$, $JSF(32^21)=0$, $JSF(42^2)=\chi(32^21)$, $JSF(4^2)=\chi(42^2)-\chi(32^21)={\rm ch}L(42^2)$,
$JSF(62)=\chi(1^2)+\chi(4^2)+\chi(42^2)={\rm ch}L(1^2)+{\rm ch}L(4^2)+2{\rm ch}L(42^2)+{\rm ch}L(32^21)$, and $JSF(642)=-\chi(1^2)+\chi(42^2)+\chi(4^2)+2\chi(62)={\rm ch}L(4^2)+2{\rm ch}L(42^2)+{\rm ch}L(32^21)+2\chi(62)-\chi(1^2)$. From the formula for $JSF(62)$ we know that $L(1^2)$ occurs in $\Delta(62)$, so $2\chi(62)-\chi(1^2)$ is the character of a $G$-module and $L(32^21)$ must occur in $\Delta(642)$.
Note that both $-\chi(1^2)$ and $\chi(42^2)$ in $JSF(642)$ come from the root $\ve_1+\ve_2$. Their $l$-values are $4$ and $5$, respectively.
\end{exa}

Note that $\lambda_1+l(\lambda)\le p$ implies that $\lambda$ is a $p$-core, since $\lambda_1+l(\lambda)-1$ is the greatest hook length. For the next result we need this stronger assumption.
Define the partial order $\preceq$ on $X^+$ as follows:

\medskip

\emph{$\mu\preceq\lambda$ if and only if there is a sequence of dominant weights $\lambda=\lambda_1,\ldots,\lambda_t=\mu$, $t\ge1$, such that for all $r\in\{1,\ldots,t-1\}$, $\lambda_{r+1}=ws_{\alpha,l}\cdot\lambda_r$ for some $w\in\Sym(\{1,\ldots,l(\lambda_r)\})$, $\alpha=\ve_i+\ve_j$, $1\le i<j\le l(\lambda_r)$, and $l\ge1$ with $\la\lambda_r+\rho,\alpha^\vee\ra-lp\ge1$, and all entries of $s_{\alpha,l}(\lambda_r+\rho)$ distinct and strictly positive.}

\medskip

Note that $\lambda_1+l(\lambda)\le p$ and $\mu\preceq\lambda$ implies that $\mu\subseteq\lambda$ and $\mu$ is $W_p(D_{l(\lambda)})$-conjugate to $\lambda$ under the dot action, which, in turn, implies that $\mu\le\lambda$. 
For $s\in\{1,\ldots,m\}$ put
$$\Lambda_s=\{\mu\in X^+\,|\,\mu_1+l(\mu)\le p\text{\ and\ }l(\mu)\le s\}\,.$$

Assertion (i) below says that, when $\lambda_1+l(\lambda)\le p$, nonzero contributions of roots $\alpha=\ve_i+\ve_j$, $1\le i<j\le l(\lambda)$ are always surviving and have a unique $l$-value.

\begin{prop}\label{prop.linkage}
Let $\lambda\in\Lambda_m$, i.e. $\lambda\in X^+$ with $\lambda_1+l(\lambda)\le p$.
\begin{enumerate}[{\rm(i)}]
\item If $\alpha=\ve_i+\ve_j$, $1\le i<j\le l(\lambda)$, and $l,a$ are integers $\ge1$ such that $\la\lambda+\rho,\alpha^\vee\ra=a+lp$ and $\chi(s_{\alpha,l}\cdot\lambda)\ne0$, then $(\lambda+\rho)_j-a>0$ and $a<p-1$.
\item If $\Lambda\subseteq\Lambda_m$ is $\preceq$-saturated, then the algebra $O_\Lambda(k[G])^*$ is quasihereditary with partially ordered labelling set $(\Lambda,\preceq)$ and the Weyl and induced modules as standard and costandard modules. In particular, if $[\Delta(\lambda):L(\mu)]$ or $(T(\lambda):\nabla(\mu))$ is nonzero, then $\mu\preceq\lambda$.
\item If $\mu$ is maximal with respect to $\preceq$ amongst the dominant weights $\nu$ for which $\chi(\nu)$ occurs in the RHS of the reduced Jantzen Sum Formula associated to $\lambda$ or amongst the dominant weights $\prec\lambda$, then we have \break $\dim\Hom_G(\nabla(\lambda),\nabla(\mu))=[\Delta(\lambda):L(\mu)]\ne0$.
\end{enumerate}
\end{prop}

\begin{proof}
(i).\ Since $s_{\alpha,l}(\lambda+\rho)_j=(\lambda+\rho)_j-a$ we cannot have $(\lambda+\rho)_j-a=0$. Now assume $(\lambda+\rho)_j-a<0$. Then $0<a-(\lambda+\rho)_j=(\lambda+\rho)_i-lp\le\rho_i+\lambda_i-p\le m-l(\lambda)$, since $\lambda_i+l(\lambda)\le \lambda_1+l(\lambda)\le p$. In particular $l(\lambda)<m$ and the last $m-l(\lambda)$ entries of $\lambda+\rho$ form an interval with smallest value $1$. Now $a-(\lambda+\rho)_j$ must occur in this interval and 
$s_{\alpha,l}(\lambda+\rho)$ contains a repeat up to sign. Contradiction. So $(\lambda+\rho)_j-a>0$. If $a\ge p-1$, then $0<(\lambda+\rho)_j-a\le(\lambda+\rho)_j-p+1\le m-l(\lambda)$, since $\rho_j\le m-1$, and $s_{\alpha,l}(\lambda+\rho)$ would contain a repeat.\\
(ii).\ First we show that $[\Delta(\lambda):L(\mu)]\ne0$ implies $\mu\preceq\lambda$. If $\mu=\lambda$, the assertion is clear. Otherwise, $L(\mu)$ must be a composition factor of ${\rm rad}\,\Delta(\lambda)$ and therefore ${\rm ch}\,L(\mu)$ must occur in the reduced sum. So it must occur in some $\chi(\nu)$ with $\nu\preceq\lambda$, $\nu\ne\lambda$. Clearly, $\nu_1+l(\nu)\le p$. So the assertion follows by induction on $|\lambda|$. Now let $\Lambda$ be as stated. Let $\Lambda_0$ be a finite $\le$-saturated subset of $X^+$ containing $\Lambda$. Define the partial order $\sqsubseteq$ on $X^+$ by $$\mu\sqsubseteq\nu\Leftrightarrow\begin{cases}\mu\preceq\nu\text{\ if\ }\nu_1+l(\nu)\le p\,,\\ \mu\le\nu\text{\ if\ }\nu_1+l(\nu)>p\,.\end{cases}$$
The algebra $O_{\Lambda_0}(k[G])^*$ is quasihereditary with partially ordered labelling set $(\Lambda_0,\le)$ and the Weyl and induced modules as standard and costandard modules.
By Lemma~\ref{lem.quasihereditary} and what we just proved, this algebra is also quasihereditary for the partial order $\sqsubseteq$. Now we observe that $\sqsubseteq$ and $\preceq$ coincide on $\Lambda_m$ and we can truncate further to $O_\Lambda(k[G])^*$. Finally, assume $(T(\lambda):\nabla(\mu))\ne0$. Then we can take $\Lambda=\{\nu\in X^+\,|\,\nu\preceq\lambda\}$ above and we obtain that $\mu\in\Lambda$, i.e. $\mu\preceq\lambda$.\\
(iii).\ If $\mu$ is  maximal amongst the dominant weights $\prec\lambda$, then, by the definition of $\preceq$, (a nonzero multiple of) $\chi(\mu)$ must occur in the reduced sum associated to $\lambda$, say $\mu={\rm sort}(s_{\alpha,l}(\lambda+\rho))-\rho$. Now write the reduced sum as a linear combination of irreducible characters. By the maximality of $\mu$ the character ${\rm ch}\,L(\mu)$ only occurs in the term $\nu_p(lp)c\chi(\mu)$, $c=\pm1$. It follows that $c=1$ and $[\Delta(\lambda):L(\mu)]\ne0$.
Put $\Lambda=\{\nu\in X^+\,|\,\nu\preceq\lambda\}$. Applying Lemma~\ref{lem.quasihereditary2} to the quasihereditary algebra $O_\Lambda(k[G])^*$ we get $\dim\Hom_G(\nabla(\lambda),\nabla(\mu))=[\nabla(\lambda):L(\mu)]=[\Delta(\lambda):L(\mu)]\ne0$.
\end{proof}

\begin{exa}
In the reduced sum different $\alpha$'s may have different $l$-values when we assume $\lambda_1+l(\lambda)\le p$. For example, when $m=5$, $p=11$ and $\lambda=(7^261)$ we have $JSF(\lambda)=\chi(6^31)+\chi(7^25)$, where the contributions come from $\ve_1+\ve_2$ with $l=2$ and $\ve_3+\ve_4$ with $l=1$. Further calculation yields $JSF(6^31)=JSF(7^25)={\rm ch}\,L(6^25)$ and $JSF(\lambda)={\rm ch}\,L(6^31)+{\rm ch}\,L(7^25)+2{\rm ch}\,L(6^25)$. There can't be more than two $l$-values in the reduced sum when $\lambda_1+l(\lambda)\le p$, see Remark~\ref{rems.preceq}.2.
\end{exa}

\section{Translation Functors}\label{s.translation}
The results in this section are analogues of \cite[Thm~3.2,3.3, Prop~3.4]{CdV} and \cite[II.7.9, 7.14-16]{Jan}. Our results don't follow from the ones in \cite{Jan}, see Remark~\ref{rems.translation}.1. One could try to reformulate/generalise these results in terms of $W_p(D_s)$ and a type $D_s$ alcove geometry similar to \cite[Sect~5-7]{CdVM2}, 
but instead we will choose a ``combinatorial" approach similar to \cite{CdV}, using the notion of the ``support" of a partition. This suffices for our applications in Sections~\ref{s.filt_mult} and \ref{s.dec_num}. Compared to \cite{CdV} the notion of the support of a partition arises from an application of Brauer's formula \cite[II.5.8]{Jan} and the role of the induction and restriction functors in \cite{CdV} is in our setting played by the translation functors.

Recall that the tensor product of two modules with a good/Weyl filtration has a good/Weyl filtration, see \cite[II.4.21, 2.13]{Jan}. Let $\lambda\in X^+$. Then we have by Brauer's formula that $\chi(\lambda)\sum_{i=1}^m(e(\ve_i)+e(-\ve_i))=\sum_{\mu\in{\rm Supp}(\lambda)}\chi(\mu)$, where ${\rm Supp}(\lambda)$ consists of all partitions of length $\le m$ which can be obtained by adding a box to $\lambda$ or removing a box from $\lambda$. Here we used the rules for $\chi(\lambda)$ to be nonzero from Section~\ref{s.JSF}.
Since ${\rm ch}\,V=\sum_{i=1}^m(e(\ve_i)+e(-\ve_i))$, it follows that $\nabla(\lambda)\ot V$ has a good filtration with sections $\nabla(\mu)$, $\mu\in{\rm Supp}(\lambda)$ and $\Delta(\lambda)\ot V$ has a Weyl filtration with sections $\Delta(\mu)$, $\mu\in{\rm Supp}(\lambda)$.

First we recall the definition and basic properties of certain ordinary translation functors. For $\lambda\in X^+$ the projection functor ${\rm pr}_\lambda:\{G\text{-modules}\}$ $\to\{G\text{-modules}\}$ is defined by ${\rm pr}_\lambda M=O_{W_p\cdot\lambda\cap X^+}(M)$. Then $M=\bigoplus_\lambda{\rm pr}_\lambda M$ where the sum is over a set of representatives of the type $C_m$ linkage classes in $X^+$, see \cite[II.7.3]{Jan}.
Now let $\lambda,\lambda'\in X^+$ with $\lambda'\in{\rm Supp}(\lambda)$. Then we define the \emph{translation functor} $T_\lambda^{\lambda'}:\{G\text{-modules}\}\to\{G\text{-modules}\}$ by $T_\lambda^{\lambda'}M={\rm pr}_{\lambda'}(({\rm pr}_\lambda M)\ot V)$.
So this is just a special case of the translation functors from \cite[II.7.6]{Jan}, since $\ve_1$ is the dominant $W$-conjugate of $\lambda'-\lambda$ and $V=\nabla(\ve_1)=L(\ve_1)=V^*$.
In particular, $T_\lambda^{\lambda'}$ is exact and left and right adjoint to $T_{\lambda'}^\lambda$.
Note that, for $\mu\in X^+\cap  W_p\cdot\lambda$, $T_\lambda^{\lambda'}\nabla(\mu)$ has a good filtration with sections $\nabla(\nu)$, $\nu\in{\rm Supp}(\mu)\cap  W_p\cdot\lambda'$, and the analogue for Weyl modules and Weyl filtrations also holds.

We will actually work with certain refined translation functors which we define now. Recall the definition of the set $\Lambda_s$ from Section~\ref{s.JSF}.
If $\Lambda\subseteq\Lambda_s$ is a $\preceq$-saturated set, then, by Proposition~\ref{prop.linkage}(ii), the type $D_s$ linkage principle holds in $\mc C_\Lambda$. So if $\lambda,\mu\in\Lambda$ belong to the same $\mc C_\Lambda$-block, then they are conjugate under the dot action of $W_p(D_s)$.
For $\lambda\in\Lambda_s$ we define the projection functor $\widetilde{\rm pr}_\lambda:\mc C_{\Lambda_s}\to\mc C_{\Lambda_s}$ by $\widetilde{\rm pr}_\lambda M=O_{W_p(D_s)\cdot\lambda\cap X^+}(M)$. Then $M=\bigoplus_\lambda\widetilde{\rm pr}_\lambda M$ where the sum is over a set of representatives of the type $D_s$ linkage classes in $\Lambda_s$. Note that $\widetilde{\rm pr}_\lambda M$ is a direct summand of ${\rm pr}_\lambda M$.
Now let $\lambda,\lambda'\in\Lambda_s$ with $\lambda'\in{\rm Supp}(\lambda)$ and let $\mc C,\mc C'$ be Serre subcategories of $\mc C_{\Lambda_s}$ with ${\rm pr}_{\lambda'}((\widetilde{\rm pr}_\lambda M)\ot V)\in\mc C_{\Lambda_s}$ for all $M\in\mc C$ and ${\rm pr}_\lambda((\widetilde{\rm pr}_{\lambda'}M)\ot V)\in\mc C_{\Lambda_s}$ for all $M\in\mc C'$. Then we define the \emph{translation functors} $\widetilde T_\lambda^{\lambda'}:\mc C\to\mc C_{\Lambda_s}$ and $\widetilde T_{\lambda'}^\lambda:\mc C'\to\mc C_{\Lambda_s}$ by $\widetilde T_\lambda^{\lambda'}M=\widetilde{\rm pr}_{\lambda'}((\widetilde{\rm pr}_\lambda M)\ot V)$ and $\widetilde T_{\lambda'}^\lambda M=\widetilde{\rm pr}_\lambda((\widetilde{\rm pr}_{\lambda'} M)\ot V)$. Note that if $\mu\in X^+\cap  W_p(D_s)\cdot\lambda$ and $\nabla(\mu)\in\mc C$, then $\widetilde T_\lambda^{\lambda'}\nabla(\mu)$ has a good filtration with sections $\nabla(\nu)$, $\nu\in{\rm Supp}(\mu)\cap  W_p(D_s)\cdot\lambda'$. The analogue for Weyl modules and Weyl filtrations also holds.
If $\widetilde T_\lambda^{\lambda'}$ and $\widetilde T_{\lambda'}^\lambda$ have image in $\mc C$ and $\mc C'$, then they restrict to functors $\mc C\to\mc C'$ and $\mc C'\to \mc C$ which are exact and each others left and right adjoint.
\begin{prop}[Translation equivalence]\label{prop.trans_equivalence}
Let $\lambda,\lambda'\in\Lambda_s$ with $\lambda'\in{\rm Supp}(\lambda)$ and let $\Lambda\subseteq W_p(D_s)\cdot\lambda\cap\Lambda_s,\Lambda'\subseteq W_p(D_s)\cdot\lambda'\cap\Lambda_s$ be $\preceq$-saturated sets. 
Assume
\begin{enumerate}[{\rm (1)}]
\item ${\rm Supp}(\nu)\cap W_p\cdot\lambda'\subseteq\Lambda_s$ for all $\nu\in\Lambda$, and ${\rm Supp}(\nu')\cap  W_p\cdot\lambda\subseteq\Lambda_s$ for all $\nu'\in\Lambda'$.
\item $|{\rm Supp}(\nu)\cap W_p(D_s)\cdot\lambda'|=1=|{\rm Supp}(\nu')\cap W_p(D_s)\cdot\lambda|$ for all $\nu\in\Lambda$ and $\nu'\in\Lambda'$.
\item The map $\nu\mapsto\nu':\Lambda\to\Lambda_s$ given by ${\rm Supp}(\nu)\cap W_p(D_s)\cdot\lambda'=\{\nu'\}$ has image $\Lambda'$, and together with its inverse $\Lambda'\to\Lambda$ it preserves the order $\preceq$.
\end{enumerate}
Then $\widetilde T_\lambda^{\lambda'}$ restricts to an equivalence of categories $\mc C_\Lambda\to\mc C_{\Lambda'}$ with inverse $\widetilde T_{\lambda'}^\lambda:\mc C_{\Lambda'}\to\mc C_\Lambda$.
Furthermore, with $\nu$ and $\nu'$ as in (3), we have $\widetilde T_\lambda^{\lambda'}\nabla(\nu)=\nabla(\nu')$, $\widetilde T_\lambda^{\lambda'}\Delta(\nu)=\Delta(\nu')$, $\widetilde T_\lambda^{\lambda'}L(\nu)=L(\nu')$, $\widetilde T_\lambda^{\lambda'}T(\nu)=T(\nu')$ and $\widetilde T_\lambda^{\lambda'}I_\Lambda(\nu)=I_{\Lambda'}(\nu')$.
\end{prop}

\begin{proof}
Note that $\widetilde{\rm pr}_\lambda=\id$ on $\mc C_\Lambda$ and $\widetilde{\rm pr}_{\lambda'}=\id$ on $\mc C_{\Lambda'}$.
From (1) we deduce that, for $M\in\mc C_{\Lambda}$, ${\rm pr}_{\lambda'}(M\ot V)\in\mc C_{\Lambda_s}$. So $\widetilde{\rm pr}_{\lambda'}(M\ot V)$ is well-defined and, by (2) and (3), it belongs to $C_{\Lambda'}$. The same holds with the roles of $\Lambda$ and $\Lambda'$ reversed, so we can apply the construction before the proposition.

The identities involving the induced and Weyl modules are now obvious. We have an exact sequence
\begin{align}\label{eq.Delta}0\to M\to\Delta(\nu)\to L(\nu)\to 0\,,\end{align}
where all composition factors $L(\eta)$ of $M$ satisfy $\eta\prec\nu$. Applying $\widetilde T_\lambda^{\lambda'}$ gives the exact sequence
\begin{align}\label{eq.T_Delta}0\to \widetilde T_\lambda^{\lambda'}M\to\Delta(\nu')\to \widetilde T_\lambda^{\lambda'}L(\nu)\to 0\,.\end{align}
Using the order preserving properties of $\nu\mapsto\nu'$ we see that for any $\theta\in\Lambda$ all composition factors $L(\eta')$ of $\widetilde T_\lambda^{\lambda'}L(\theta)$ satisfy $\eta'\preceq\theta'$. So all composition factors $L(\eta')$ of $\widetilde T_\lambda^{\lambda'}M$ satisfy $\eta'\prec\nu'$.
Therefore $\widetilde T_\lambda^{\lambda'}L(\nu)$ must have simple head $L(\nu')$ and all other composition factors $L(\eta')$ satisfy $\eta'\prec\nu'$. If $\widetilde T_\lambda^{\lambda'}L(\nu)\ne L(\nu')$, then $$\Hom_G(\Delta(\eta),L(\nu))=\Hom_G(\widetilde T_{\lambda'}^\lambda\Delta(\eta'),L(\nu))=\Hom_G(\Delta(\eta'),\widetilde T_\lambda^{\lambda'}L(\nu))\ne0$$
for some $\eta\ne\nu$. This is clearly impossible, so $\widetilde T_\lambda^{\lambda'}L(\nu)=L(\nu')$. We can prove the same for $\widetilde T_{\lambda'}^\lambda$, and then we can deduce as in the proof \cite[II.7.9]{Jan} that $\widetilde T_{\lambda'}^\lambda\widetilde T_\lambda^{\lambda'}\cong\id_{\mc C_\Lambda}$ and $\widetilde T_\lambda^{\lambda'}\widetilde T_{\lambda'}^\lambda\cong\id_{\mc C_{\Lambda'}}$. This implies the remaining assertions.
\end{proof}

\begin{prop}[Translation projection]\label{prop.trans_projection}
Let $\lambda,\lambda'\in\Lambda_s$ with $\lambda'\in{\rm Supp}(\lambda)$ and let $\Lambda\subseteq W_p(D_s)\cdot\lambda\cap\Lambda_s,\Lambda'\subseteq W_p(D_s)\cdot\lambda'\cap\Lambda_s$ be $\preceq$-saturated sets. Put $\tilde\Lambda=\{\nu\in\Lambda\,|\,{\rm Supp}(\nu)\cap W_p(D_s)\cdot\lambda'\ne\emptyset\}$. Assume
\begin{enumerate}[{\rm (1)}]
\item ${\rm Supp}(\nu)\cap W_p\cdot\lambda'\subseteq\Lambda_s$ for all $\nu\in\Lambda$, and ${\rm Supp}(\nu')\cap W_p\cdot\lambda\subseteq\Lambda_s$ for all $\nu'\in\Lambda'$.
\item $|{\rm Supp}(\nu)\cap W_p(D_s)\cdot\lambda'|=1$ for all $\nu\in\tilde\Lambda$, and $|{\rm Supp}(\nu')\cap W_p(D_s)\cdot\lambda|=2$ for all $\nu'\in\Lambda'$.
\item The map $\nu\mapsto\nu':\tilde\Lambda\to\Lambda_s$ given by ${\rm Supp}(\nu)\cap W_p(D_s)\cdot\lambda'=\{\nu'\}$ is a 2-to-1 map which has image $\Lambda'$ and preserves the order $\preceq$. For $\nu'\in\Lambda'$ we can write ${\rm Supp}(\nu')\cap W_p(D_s)\cdot\lambda=\{\nu^+,\nu^-\}$ with $\nu^-\prec\nu^+$ and then $\Hom_G(\nabla(\nu^+),\nabla(\nu^-))\ne0$ and $\eta'\preceq\nu'\Rightarrow\eta^+\preceq\nu^+$ and $\eta^-\preceq\nu^-$.
\end{enumerate}
Then $\widetilde T_\lambda^{\lambda'}$ restricts to a functor $\mc C_\Lambda\to\mc C_{\Lambda'}$ and $\widetilde T_{\lambda'}^\lambda$ restricts to a functor $\mc C_{\Lambda'}\to\mc C_\Lambda$.
Now let $\nu\in\Lambda$. If $\nu\notin\tilde\Lambda$, then $\widetilde T_\lambda^{\lambda'}\nabla(\nu)=\widetilde T_\lambda^{\lambda'}\Delta(\nu)=\widetilde T_\lambda^{\lambda'}L(\nu)=0$.
For $\nu'\in\Lambda'$ with $\nu^\pm$ as in (3), we have $\widetilde T_\lambda^{\lambda'}\nabla(\nu^\pm)=\nabla(\nu')$, $\widetilde T_\lambda^{\lambda'}\Delta(\nu^\pm)=\Delta(\nu')$, $\widetilde T_\lambda^{\lambda'}L(\nu^-)=L(\nu')$, $\widetilde T_\lambda^{\lambda'}L(\nu^+)=0$, $\widetilde T_{\lambda'}^\lambda T(\nu')=T(\nu^+)$ and $\widetilde T_{\lambda'}^\lambda I_{\Lambda'}(\nu')=I_\Lambda(\nu^-)$.
\end{prop}
\begin{proof}
The first assertion follows as in the proof of Proposition~\ref{prop.trans_equivalence} and all identities involving the induced and Weyl modules are obvious. Moreover, it is also clear that $\widetilde T_\lambda^{\lambda'}L(\nu)=0$ when $\nu\notin\tilde\Lambda$, since $\widetilde T_\lambda^{\lambda'}\Delta(\nu)$ surjects onto $\widetilde T_\lambda^{\lambda'}L(\nu)$ and is $0$. If $\eta\prec\nu^-$, then $\eta'\prec\nu'$, 
so we obtain $\widetilde T_\lambda^{\lambda'}L(\nu^-)=L(\nu')$ as in the proof of Proposition~\ref{prop.trans_equivalence}.
Now consider \eqref{eq.Delta} and \eqref{eq.T_Delta} for $\nu=\nu^+$. Since $[\Delta(\nu^+):L(\nu^-)]=[\nabla(\nu^+):L(\nu^-)]\ne0$, we know that $L(\nu^-)$ occurs in $M$. So $\widetilde T_\lambda^{\lambda'}L(\nu^-)=L(\nu')$ occurs in $\widetilde T_\lambda^{\lambda'}M$ and therefore not in $\widetilde T_\lambda^{\lambda'}L(\nu^+)$.
If $\widetilde T_\lambda^{\lambda'}L(\nu^+)\ne0$, then it would have simple head $L(\nu')$ by \eqref{eq.T_Delta}. So $\widetilde T_\lambda^{\lambda'}L(\nu^+)=0$.
Note that ${\rm ch}\,\widetilde T_\lambda^{\lambda'}\widetilde T_{\lambda'}^\lambda M=2{\rm ch}\,M$ for any $M\in\mc C_{\Lambda'}$ which has a good or Weyl filtration. Now the equality $\widetilde T_{\lambda'}^\lambda T(\nu')=T(\nu^+)$ is proved as in \cite[E.11]{Jan}, replacing $\uparrow, w\cdot\lambda, ws\cdot\lambda, w\cdot\mu, T_\lambda^\mu$ and $T_\mu^\lambda$ by $\preceq, \nu^+, \nu^-, \mu, \widetilde T_\lambda^{\lambda'}$ and $\widetilde T_{\lambda'}^\lambda$. Finally,
\begin{align}\label{eq.I}\Hom_{\mc C_\Lambda}(-,\widetilde T_{\lambda'}^\lambda I_{\Lambda'}(\nu'))=\Hom_{\mc C_{\Lambda'}}(-,I_{\Lambda'}(\nu'))\circ \widetilde T_\lambda^{\lambda'}\end{align}
is exact, so $\widetilde T_{\lambda'}^\lambda I_{\Lambda'}(\nu')$ is injective in $\mc C_\Lambda$. Applying both sides of \eqref{eq.I} to $L(\eta)$, for $\eta\notin\tilde\Lambda$, for $\eta=\eta^+$ and for $\eta=\eta^-$, shows that $\widetilde T_{\lambda'}^\lambda I_{\Lambda'}(\nu')$ has simple socle $L(\nu^-)$ and therefore equals $I_\Lambda(\nu^-)$.
\end{proof}

\begin{rems}\label{rems.translation}
1. The translated weight $\lambda'$ need not be in the facet closure of $\lambda$. For example, when $p=5$, $m=7$, $s=2$ and $(\lambda,\lambda')=((2^2),(21))$ or $((21),(1^2))$, then it is easy to find affine reflection hyperplanes which contain $\lambda$, but not $\lambda'$. 
However, we can, for $\Lambda=\{(2^2),\emptyset\}$ and $\Lambda'=\{(21),(1)\}$, apply Proposition~\ref{prop.trans_equivalence} in the first case, and, for $\Lambda=\{(21),(1)\}$ and $\Lambda'=\{(1^2)\}$, apply Proposition~\ref{prop.trans_projection} in the second case. We refer to Section~\ref{s.arrow_diagrams} for how to express this in terms of arrow diagrams.\\
2. As we will see later, the use of the type $D_s$ linkage principle is only needed for moves from the $0$-node. For most pairs $(\lambda,\lambda')$ we could just use the usual type $C_m$ linkage principle, i.e. the usual translation functors. So we use the refined translation functors to be able to deal with any move.
\end{rems}

\section{Arrow diagrams}\label{s.arrow_diagrams}
This section is based on the approaches of \cite{CdV} and \cite{Sh}. We use the ``characteristic $p$ wall" of \cite{Sh}, but at the same time we use the transposed labels (from the Brauer algebra point of view) as in \cite{CdV}. Recall the definition of $\rho$ from Section~\ref{s.prelim}.
An arrow diagram has $(p+1)/2$ \emph{nodes} on a (horizontal) \emph{line} with $p$ \emph{labels}: $0$ and $\pm i$, $i\in\{1,\ldots,(p-1)/2\}$. The $i$-th node from the left has top label $-(i-1)$ and a bottom label $i-1$. So the first node is the only node whose top and bottom label are the same. Next we choose $s\in\{1,\ldots,\min(m,p)\}$ and put a wall between $\rho_s$ and $\rho_s-1$ mod $p$. So when $\rho_s=(p+1)/2$ mod $p$ we can put the wall above or below the line, otherwise there is only one possibility.
Then we can also put in the \emph{values}, one for each label. A value and its corresponding label are always equal mod $p$. We start with $\rho_s$ immediately after the wall in the anti-clockwise direction, and then increasing in steps of $1$ going in the anti-clockwise direction around the line: $\rho_s,\rho_s+1,\ldots,\rho_s+p-1$.
For example, when $p=5$, $m=7$ and $s=2$, then $\rho_s=6$ and we have labels
$$
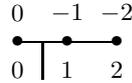
\resizebox{1.6cm}{.5cm}{\xy
(0,0)="a1"*{\bullet};(0,4.5)*{0};(0,-4.5)*{0};(3.8,-3)*{\rule[0mm]{.3mm}{6mm}};
(7.5,0)="a2"*{\bullet};(7.5,4.5)*{-1};(7.5,-4.5)*{1};
(15,0)="a3"*{\bullet};(15,4.5)*{-2};(15,-4.5)*{2};
{"a1";"a3"**@{-}}; 
\endxy}
$$
(usually we omit the top labels), and values
$$
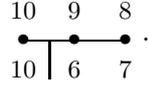
\resizebox{1.6cm}{.5cm}{\xy
(0,0)="a1"*{\bullet};(0,4.5)*{10};(0,-4.5)*{10};(3.8,-3)*{\rule[0mm]{.3mm}{6mm}};
(7.5,0)="a2"*{\bullet};(7.5,4.5)*{9};(7.5,-4.5)*{6};
(15,0)="a3"*{\bullet};(15,4.5)*{8};(15,-4.5)*{7};
{"a1";"a3"**@{-}}; 
\endxy}\ .
$$
For a partition $\lambda$ with $l(\lambda)\le s\le p-\lambda_1$ we now form the ($s$-)\emph{arrow diagram} by putting in $s$ arrows ($\vee$ or $\land$) that point \emph{from} the values $(\rho+\lambda)_1,\ldots,(\rho+\lambda)_s$, or the corresponding labels. In case of the label $0$ we have two choices for the arrow. So in the above example the arrow diagram of $\lambda=(1^2)$ is
$$
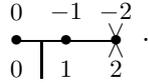
\resizebox{1.6cm}{.5cm}{\xy
(0,0)="a1"*{\bullet};(0,4.5)*{0};(0,-4.5)*{0};(3.8,-3)*{\rule[0mm]{.3mm}{6mm}};
(7.5,0)="a2"*{\bullet};(7.5,4.5)*{-1};(7.5,-4.5)*{1};
(15,0)="a3"*{\bullet};(15,1.5)*{\vee};(15,-1.5)*{\land};(15,4.5)*{-2};(15,-4.5)*{2};
{"a1";"a3"**@{-}}; 
\endxy}\ .
$$
In such a diagram we frequently omit the nodes and/or the labels. When it has already been made clear what the labels are and where the wall is, we can simply represent the arrow diagram by a string of single arrows ($\land$, $\vee$), opposite pairs of arrows ($\times$) and symbols ${\rm o}$ to indicate the absence of an arrow. In the above example $\lambda=(1^2)$ is then represented by ${\rm oo}\times$ and $\lambda=(32)$ is represented by $\vee{\rm o}\vee$ or $\land{\rm o}\vee$.

We can form the arrow diagram of $\lambda$ by first lining all $s$ arrows up against the wall and then moving them in the anticlockwise direction to the right positions. The arrow furthest from the wall (in the anti-clockwise direction) corresponds to $\lambda_1$, and the arrow closest to the wall corresponds to $\lambda_s$. The part corresponding to an arrow equals the number of labels without an arrow from that arrow to the wall in the clockwise direction. From the diagram you can see what you can do with the wall, changing $s$ but not $\lambda$: If there is an arrow immediately after the wall in the anti-clockwise direction, i.e. $l(\lambda)<s$, then you can move the wall one step in the anti-clockwise direction, removing the arrow that you move it past. If there is no arrow immediately after the wall in the clockwise direction, i.e. $\lambda_1<p-s$, then you can move the wall one step in the clockwise direction, putting an arrow at the label that you move it past, provided $s<m$.

More generally, we can for any $s\in\{1,\ldots,m\}$ and $\mu\in X^+$ with $l(\mu)\le s$, put $s$ arrows in the diagram pointing from the labels equal to $(\rho+\mu)_1,\ldots,(\rho+\mu)_s$ mod $p$, allowing repeated arrows at a label. Then $\mu$ and $\nu$ with $l(\mu),l(\nu)\le s$ are $W_p(C_s)$-conjugate under the dot action if and only if $|\mu|-|\nu|$ is even and the arrow diagram of $\nu$ can be obtained from that of $\mu$ by repeatedly replacing an arrow by its opposite, i.e. if and only if $|\mu|-|\nu|$ is even and the arrow diagrams of $\mu$ and $\nu$ have the same number of arrows at each node. Furthermore, $\mu$ and $\nu$ with $l(\mu),l(\nu)\le s$ are $W_p(D_s)$-conjugate under the dot action if and only if $|\mu|-|\nu|$ is even and the arrow diagram of $\nu$ can be obtained from that of $\mu$ by repeatedly replacing two arrows by their opposites, and possibly replacing an arrow with label $0$ by its opposite.

From now on $s\in\{1,\ldots,\min(m,p)\}$, unless stated otherwise. We put
$$\Lambda(s)=\{\lambda\in X^+\,|\,l(\lambda)\le s\le p-\lambda_1\}\,.$$
Unless stated otherwise, we assume $\lambda\in\Lambda(s)$.

When we speak of ``arrow pairs" it is understood that both arrows are single, i.e. neither of the two arrows is part of an $\times$. So, for example, at the node of the first arrow in an arrow pair $\vee\land$ there should not also be a $\land$. The arrows need not be consecutive in the diagram.

We now define the \emph{cap-curl diagram} $c_\lambda$ of the arrow diagram associated to $\lambda$ as follows. All caps and curls are anti-clockwise, starting from the arrow closest to the wall. We start on the left side of the wall. We first form the caps recursively. Find an arrow pair $\vee\land$ that are neighbours in the sense that the only arrows in between are already connected with a cap or are part of an $\times$, and connect them with a cap. Repeat this until there are no more such arrow pairs.
Now the unconnected arrows that are not part of an $\times$ form a sequence $\land\cdots\land\vee\cdots\vee$. We connect consecutive (in the mentioned sequence) $\land\land$ pairs with a curl, starting from the left. At the end the unconnected arrows that are not part of an $\times$ form a sequence $\land\vee\cdots\vee$ or just a sequence of $\vee$'s. Note that none of these arrows occur inside a cap or curl.
The caps on the right side of the wall are formed in the same way. The curls now connect consecutive $\vee\vee$ pairs and are formed starting from the right. So at the end the unconnected arrows that are not part of an $\times$ form a sequence $\land\cdots\land\vee$ or just a sequence of $\land$'s. Again, none of these arrows occur inside a cap or curl.
For example, when $p=23$, $m=17$, $s=12$ and $\lambda=(11,11,11,11,11,11,10,6,4,4,1)$, then
$c_\lambda$ is
$$
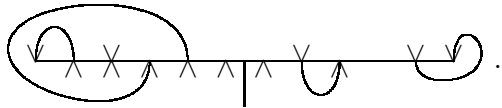
\xy
(0,0)="a1";(0,1)*{\vee};
(5,0)="a2";(5,-1)*{\land};
(10,0)="a3";(10,1)*{\vee};(10,-1)*{\land};
(15,0)="a4";(15,-1)*{\land};
(20,0)="a5";(20,-1)*{\land};
(25,0)="a6";(25,-1)*{\land};(27.5,-3)*{\rule[0mm]{.3mm}{6mm}};
(30,0)="a7";(30,-1)*{\land};
(35,0)="a8";(35,1)*{\vee};
(40,0)="a9";(40,-1)*{\land};
(45,0)="a10";
(50,0)="a11";(50,1)*{\vee};
(55,0)="a12";(55,1)*{\vee};
{"a1";"a12"**@{-}}; 
"a2";"a1"**\crv{(5,6)&(0,6)}; 
"a9";"a8"**\crv{(40,-6)&(35,-6)}; 
"a5";"a4"**\crv{(20,8)&(-5,9)&(-5,-5)&(15,-7)}; 
"a11";"a12"**\crv{(50,-3)&(59,-3)&(59,4)&(55,4)}; 
\endxy\ .$$
Note that the $10$-th node which has labels $\pm9$ and values $9$ and $14$, has no arrow.

\begin{lem}\label{lem.JSF-arrows}
Let $\lambda\in\Lambda(s)$.
\begin{enumerate}[{\rm(i)}]
\item The nonzero terms in the reduced Jantzen Sum Formula associated to $\lambda$ correspond in the arrow diagram of $\lambda$ to the arrow pairs $\vee\land$ to the left or to the right of the wall, and the arrow pairs $\land\land$ to the left of the wall, and the arrow pairs $\vee\vee$ to the right of the wall.
\item $\Delta(\lambda)$ is irreducible (equivalently, $T(\lambda)=\Delta(\lambda)$ or $\nabla(\lambda)$) if and only if there are no caps or curls in $c_\lambda$.
\item If $\mu$ is obtained from $\lambda$ by reversing the arrows in a pair as in (i) where the arrows are consecutive (no single arrows in between), and there are no single arrows to the left of a $\land\land$ or to the right of a $\vee\vee$, then we have $\dim\Hom_G(\nabla(\lambda),\nabla(\mu))=[\Delta(\lambda):L(\mu)]\ne0$.
\end{enumerate}
\end{lem}

\begin{proof}
(i). Write $\rho_s=x_s+up$ with $|x_s|\le(p-1)/2$ and $u\ge0$. If the wall is above the line ($x_s\le0$, $u\ge1$) the general form of a value is as indicated in the diagram below
$$
\begin{smallmatrix}
-x+up&\rule[-1mm]{.3mm}{3mm}&-x+(u+1)p\vspace{.7mm}\\
\hline\\
x+up& &x+up
\end{smallmatrix}
$$
If the wall is below the line ($x_s\ge0$) the general form of a value is as indicated in the diagram below
$$
\begin{smallmatrix}
-x+(u+1)p& &-x+(u+1)p\vspace{.7mm}\\
\hline\\
x+(u+1)p&\rule[-.7mm]{.3mm}{3mm}&x+up
\end{smallmatrix}
$$
Here $x$ always satisfies $0\le x\le(p-1)/2$. Note that the ``opposite" value on the other side of the line has the same $x$ in its general form.
Put differently, the label corresponding to the value is $x$ if the value is below the line and $-x$ if it is above the line.

Now let $\alpha=\ve_i+\ve_j$, $1\le i<j\le l(\lambda)$, and $l,a\ge1$ such that $\la\lambda+\rho,\alpha^\vee\ra=a+lp$ and $\chi(s_{\alpha,l}\cdot\lambda)\ne0$.
Put $c=(\lambda+\rho)_i$ and $d=(\lambda+\rho)_j$. Note that $c$ and $d$ cannot be opposite, because then we would have $a=0$. Assume the wall is below the line.
Then the 12 candidate configurations of $c$ and $d$ in the arrow diagram of $\lambda$ are:
$$
%
%
\begin{smallmatrix}
c\ d& &\ \vspace{.7mm}\\ 
\hline\\
\ &\rule{.3mm}{2mm}&
\end{smallmatrix}\,,
\begin{smallmatrix}
c& &d\vspace{.7mm}\\
\hline\\
\ &\rule{.3mm}{2mm}&
\end{smallmatrix}\,,
\begin{smallmatrix}
\ & &c\ d\vspace{.7mm}\\
\hline\\
\ &\rule{.3mm}{2mm}&
\end{smallmatrix}\,,
\begin{smallmatrix} 
\ & &\ \vspace{.7mm}\\
\hline\\
d\ c&\rule{.3mm}{2mm}&
\end{smallmatrix}\,,
\begin{smallmatrix}
\ & &\ \vspace{.7mm}\\
\hline\\
c&\rule{.3mm}{2mm}&d
\end{smallmatrix}\,,
\begin{smallmatrix}
\ & &\ \vspace{.7mm}\\
\hline\\
\ &\rule{.3mm}{2mm}&d\ c
\end{smallmatrix}\,,
%
%
\begin{smallmatrix} 
d\ \ & &\ \vspace{.7mm}\\
\hline\\
\ \ c&\rule[-.5mm]{.3mm}{2mm}&
\end{smallmatrix}\,,
\begin{smallmatrix}
\ \ d& &\ \vspace{.7mm}\\
\hline\\
c\ \ &\rule[-.5mm]{.3mm}{2mm}&
\end{smallmatrix}\,,
\begin{smallmatrix} 
c& &\ \vspace{.7mm}\\
\hline\\
\ &\rule{.3mm}{2mm}&d
\end{smallmatrix}\,,
\begin{smallmatrix}
\ & &d\vspace{.7mm}\\
\hline\\
c&\rule[-.5mm]{.3mm}{2mm}&
\end{smallmatrix}\,,
\begin{smallmatrix} 
\ & &c\ \ \vspace{.7mm}\\
\hline\\
\ &\rule{.3mm}{2mm}&\ \ d
\end{smallmatrix}\,,
\begin{smallmatrix}
\ & &\ \ c\vspace{.7mm}\\
\hline\\
\ &\rule{.3mm}{2mm}&d\ \ 
\end{smallmatrix}\,.
$$
Here it is understood that the opposite values of $c$ and $d$ are not present in the diagram of $\lambda+\rho$, since otherwise $s_{\alpha,l}(\lambda+\rho)$ would contain a repeat and $\chi(s_{\alpha,l}\cdot\lambda)$ would be $0$.
Now it is easy to see that the only possible configurations are 3,4,7 and 11:
$\begin{smallmatrix}
\ & &c\ d\vspace{.7mm}\\
\hline\\
\ &\rule{.3mm}{2mm}&
\end{smallmatrix}\,,
\begin{smallmatrix} 
\ & &\ \vspace{.7mm}\\
\hline\\
d\ c&\rule{.3mm}{2mm}&
\end{smallmatrix}\,,
\begin{smallmatrix} 
d\ \ & &\ \vspace{.7mm}\\
\hline\\
\ \ c&\rule[-.5mm]{.3mm}{2mm}&
\end{smallmatrix}\,,
\begin{smallmatrix} 
\ & &c\ \ \vspace{.7mm}\\
\hline\\
\ &\rule{.3mm}{2mm}&\ \ d
\end{smallmatrix}\,,\vspace{.5mm}
$
which correspond precisely to the arrow pairs from the assertion. 
For example, for configuration 1 we have $c=-x+(u+1)p, d=-y+(u+1)p$ with $0\le x<y\le(p-1)/2$. So $a=p-(x+y)$, $l=2u+1$, and $s_{\alpha,l}(\lambda+\rho)$ equals $y+up$ in position $i$ and $x+up$ in position $j$. However, the available values for the labels $x,y$ are $x+(u+1)p$ and $y+(u+1)p$. So this configuration is not possible.
Similarly, for configuration 8 we have $c=x+(u+1)p, d=-y+(u+1)p$ with $0\le x<y\le(p-1)/2$. So $a=p-(y-x)$, $l=2u+1$, and $s_{\alpha,l}(\lambda+\rho)$ equals $y+up$ in position $i$ and $-x+up$ in position $j$. However, the available values for the labels $-x,y$ are $-x+(u+1)p$ and $y+(u+1)p$. So this configuration is not possible.
As a final example, for configuration 9 we have $c=-x+(u+1)p, d=y+up$ with $0\le x<y\le(p-1)/2$. So $a=y-x$, $l=2u+1$, and $s_{\alpha,l}(\lambda+\rho)$ equals $-y+(u+1)p$ in position $i$ and $x+up$ in position $j$. However, the available values for the labels $x,-y$ are $x+(u+1)p$ and $-y+(u+1)p$. So this configuration is not possible.
The case when the wall is above the line is completely analogous.

Conversely, it is clear that if $(\alpha,l)$ corresponds to one of the stated pairs, then the entries of $s_{\alpha,l}(\lambda+\rho)$ are distinct and strictly positive, so $\chi(s_{\alpha,l}\cdot\lambda)\ne0$.\\
(ii). This follows easily from (i). For example, there is an arrow pair $\vee\land$ to the left of the wall if and only if there is a cap to the left of the wall in $c_\lambda$ (although there will in general be more such pairs than such caps).\\
(iii). Such a $\mu$ is maximal amongst the weights $\nu$ for which (a nonzero multiple of) $\chi(\nu)$ occurs on the RHS of the reduced Jantzen Sum Formula, so this follows from Proposition~\ref{prop.linkage}(iii).
\end{proof}

\begin{rems}\label{rems.preceq}
1.\ Let $s\in\{1,\ldots,\min(m,p)\}$ and let $\lambda\in\Lambda(s)$ and $\mu\in X^+$. Then it follows from the above lemma that $\mu\preceq\lambda$ if and only if $\mu\in\Lambda(s)$ and the arrow diagram of $\mu$ can be obtained from that of $\lambda$ by repeatedly replacing an arrow pair $\vee\land$ to the left or to the right of the wall, or an arrow pair $\land\land$ to the left of the wall, or an arrow pair $\vee\vee$ to the right of the wall, by the opposite arrow pair, and possibly replacing an arrow with label $0$ by its opposite.

Furthermore, $\lambda,\mu\in\Lambda(s)$ are conjugate under the dot action of $W_p(D_s)$ if and only if the arrow diagram of $\mu$ is obtained from that of $\lambda$ by replacing an even number of single arrows to the left of the wall and an even number of single arrows to the right of the wall by their opposites, and possibly replacing an arrow with label $0$ by its opposite.
Finally, $\lambda,\mu\in\Lambda(s)$ are conjugate under the dot action of $W_p$ if and only if the arrow diagram of $\mu$ is obtained from that of $\lambda$ by replacing a number of single arrows to the left of the wall and an even number of single arrows to the right of the wall by their opposites.
This follows from the fact that replacing an arrow on the left side of the wall by its opposite preserves the parity of the coordinate sum and replacing an arrow on the right side of the wall by its opposite changes the parity.\\
2.\ The $l$-values corresponding to the configurations 3,4,7 and 11 from the proof are $2u+1,2u+2,2u+2,2u+1$.
The possible configurations when the wall is above the line are:
$\begin{smallmatrix}
\ &\rule[-.5mm]{.3mm}{2mm}&c\ d\vspace{.7mm}\\
\hline\\
\ & &
\end{smallmatrix}\,,
\begin{smallmatrix} 
\ &\rule[-.5mm]{.3mm}{2mm}&\ \vspace{.7mm}\\
\hline\\
d\ c& &
\end{smallmatrix}\,,
\begin{smallmatrix} 
d\ \ &\rule[-.5mm]{.3mm}{2mm}&\ \vspace{.7mm}\\
\hline\\
\ \ c& &
\end{smallmatrix}\,,
\begin{smallmatrix} 
\ &\rule[-.5mm]{.3mm}{2mm}&c\ \ \vspace{.7mm}\\
\hline\\
\ & &\ \ d
\end{smallmatrix}\,,\vspace{.7mm}
$
with $l$-values $2u+1,2u,2u,2u+1$.
So in the reduced Jantzen Sum Formula associated to $\lambda$ we only have two possible $l$-values.

We don't know of any examples where $p$ satisfies the assumption $p>|\lambda|$ from \cite{Sh} with more than one $l$-value in the reduced sum.
In particular, we don't know of any examples of $\lambda$'s with $p>|\lambda|$ such that there is a cap or curl on the left of the wall and a cap or curl on the right of the wall in $c_\lambda$.
\end{rems}

\section{Weyl filtration multiplicities in tilting modules}\label{s.filt_mult}
Recall the definition of the set $\Lambda(s)$ from Section~\ref{s.arrow_diagrams}. Let $s\in\{1,\ldots,\min(m,p)\}$, let $\lambda\in\Lambda(s)$, and let $\mu\in X^+$ with $\mu\preceq\lambda$. Then the arrow diagram of $\mu$ has its single arrows and its $\times$'s at the same nodes as the arrow diagram of $\lambda$. If the arrow diagram of $\lambda$ has an arrow at $0$, then we assume that the parity of the number of $\land$'s in the arrow diagram of $\mu$ is the same as that for $\lambda$.\footnote{Then the parity of the number of $\vee$'s in the arrow diagram of $\mu$ is of course also the same as that for $\lambda$.} This only requires a possible change of an arrow at $0$ to its opposite in the arrow diagram of $\mu$. If there is no arrow at $0$, then these parities will automatically be the same, since $\mu$ is $W_p(D_{l(\lambda)})$-conjugate to $\lambda$ under the dot action.
Then we know, by Remark~\ref{rems.preceq}.1, that the arrow diagram of $\mu$ can be obtained from that of $\lambda$ by repeatedly replacing an arrow pair $\vee\land$ to the left or to the right of the wall, or an arrow pair $\land\land$ to the left of the wall, or an arrow pair $\vee\vee$ to the right of the wall, by the opposite arrow pair.

Recall the definition of the cap-curl diagram $c_\lambda$ from the previous section. We now define the \emph{cap-curl diagram} $c_{\lambda\mu}$ associated to $\lambda$ \emph{and $\mu$} by replacing each arrow in $c_\lambda$ by the arrow from the arrow diagram of $\mu$ at the same node. Put differently, we put the caps and curls from $c_\lambda$ on top of the arrow diagram of $\mu$. We say that $c_{\lambda\mu}$ is \emph{oriented} if all caps and curls in $c_{\lambda\mu}$ are oriented (clockwise or anti-clockwise). It is not hard to show that when $c_{\lambda\mu}$ is oriented, the arrow diagrams of $\lambda$ and $\mu$ are the same at the nodes which are not endpoints of a cap or a curl in $c_\lambda$.

For example, when $p=11$, $m=7$, $s=5$ and $\lambda=(6^332)$. Then $\rho_s=3$ and $c_\lambda$ is
$$\xy
(0,0)="a1";(0,1)*{\vee};
(5,0)="a2";(5,-1)*{\land};
(10,0)="a3";(10,-1)*{\land};(12.5,-3)*{\rule[0mm]{.3mm}{6mm}};
(15,0)="a4";
(20,0)="a5";(20,1)*{\vee};
(25,0)="a6";(25,-1)*{\land};
{"a1";"a6"**@{-}}; 
"a2";"a1"**\crv{(5,6)&(0,6)}; 
"a6";"a5"**\crv{(25,-6)&(20,-6)} 
\endxy\ .$$
The $\mu\in X^+$ with $\mu\prec\lambda$ are $(6^321), (65^232), (65^221), (5^2432), (5^2421), (4^332),$\\ $(4^321)$,
with arrow diagrams
$\vee\land\land\,{\rm o}\land\vee,\land\vee\land\,{\rm o}\vee\land, \land\vee\land\,{\rm o}\land\vee,\land\land\vee\,{\rm o}\vee\land,$
$\land\land\vee\,{\rm o}\land\vee,\vee\vee\vee\,{\rm o}\vee\land, \vee\vee\vee\,{\rm o}\land\vee$.
Only for the first three $c_{\lambda\mu}$ is oriented.

\begin{thm}\label{thm.filt_mult}
Let $s\in\{1,\ldots,\min(m,p)\}$, $\lambda\in\Lambda(s)$ and $\mu\in X^+$. Then
$$(T(\lambda):\nabla(\mu))=(T(\lambda):\Delta(\mu))=
\begin{cases}
1\text{\ if $\mu\preceq\lambda$ and $c_{\lambda\mu}$ is oriented,}\\
0\text{\ otherwise.}
\end{cases}
$$
\end{thm}
\begin{proof}
By Proposition~\ref{prop.linkage}(ii) we may assume $\mu\preceq\lambda$. The strategy of the proof is similar to that of the proof in \cite[Sect~6]{CdV}. The proof is by induction on the number of caps and curls in $c_\lambda$. If there are no caps or curls in $c_\lambda$, then $c_{\lambda\mu}$ is oriented if and only if $\lambda=\mu$, so the result follows from Lemma~\ref{lem.JSF-arrows}(ii). Otherwise, we choose a cap or curl which has no cap or curl inside it. We will transform the cap or curl to a cap for which the end nodes are consecutive via a sequence of moves which preserve the orientedness of $c_{\lambda\mu}$ and the multiplicity $(T(\lambda):\Delta(\mu))$. In the case of a cap we move the end node closest to the wall one step towards the other end node. In the case of a curl we will move the end node furthest away from the wall to the end (left or right) and then turn it into a cap. In the proof below we will make use of two basic facts. Let $t\in\{1,\ldots,m\}$. Firstly, if $\nu\in X^+$ with $l(\nu)\le t$ and $\nu'\in{\rm Supp}(\nu)$ with $l(\nu')\le t$, then the $t$-arrow diagram of $\nu'$ is obtained from that of $\nu$ by moving one arrow in the $t$-arrow diagram of $\nu$ one step. Secondly, if $\nu\in X^+$ and $\nu'\in X^+\cap W_p\cdot\nu$ have length $\le t$, then the $t$-arrow diagrams of $\nu$ and $\nu'$ have the same number of arrows at each node. See also Remark~\ref{rems.preceq}.1 for a version for the $W_p(D_s)$-action.

First we prove a general property of the moves we will make. Let $\lambda\in\Lambda(s)$ and $\lambda'\in{\rm Supp}(\lambda)\cap\Lambda(s)$ such that the move $\lambda\mapsto\lambda'$ does not cross or pass the wall. Now let $\nu\in\Lambda(s)\cap W_p\cdot\lambda$ and $\nu'\in{\rm Supp}(\nu)\cap W_p\cdot\lambda'$. We show that $\nu'\in\Lambda(s)$. The move from the arrow diagram of $\nu$ to that of $\nu'$ goes between the same nodes as the move $\lambda\mapsto\lambda'$, or between the values at the last node if this was true for $\lambda\mapsto\lambda'$. Assume $l(\nu')=s+1$. Then $l(\nu)=s<m$ and there is no arrow immediately after the wall in the anti-clockwise direction. We temporarily move the wall one step in the clockwise direction creating a new arrow immediately after the new wall in the anti-clockwise direction.\footnote{At the label of the new arrow there may be one other arrow and there may be a cap or curl of $c_\lambda$ passing or crossing the new wall.} The move from the arrow diagram of $\nu$ to that of $\nu'$ would move this new arrow one step in the anti-clockwise direction and therefore cross the original wall. But then the move $\lambda\mapsto\lambda'$ would also cross or pass the original wall. 
This is impossible, so $l(\nu')\le s$. If $\nu'_1=p-s+1$, then $\nu_1=p-s$ and the move $\nu\mapsto\nu'$ would pass or cross the wall. This would then also hold for the move $\lambda\mapsto\lambda'$ which is impossible. We conclude that $\nu'\in\Lambda(s)$.

First assume the cap or curl is to the left of the wall and assume it is a cap. If $\lambda=
\xy
(0,0)="a1";(0,0)*{\cdots};
(5,0)="a2";(5,1)*{\vee};
(10,0)="a3";(10,0)*{\cdots};
(15,0)="a4";(15,0)*{\bullet};
(20,0)="a5";(20,-1)*{\land};(20,-4)*{a};
(25,0)="a6";(25,0)*{\cdots};(25,-6.5)*{\ }; 
"a5";"a2"**\crv{(20,6)&(5,8)}; 
\endxy\ ,$
$a>1$, we choose
$\lambda'=
\xy
(0,0)="a1";(0,0)*{\cdots};
(5,0)="a2";(5,1)*{\vee};
(10,0)="a3";(10,0)*{\cdots};
(15,0)="a4";(15,-1)*{\land};
(20,0)="a5";(20,0)*{\bullet};(20,-3)*{a};
(25,0)="a6";(25,0)*{\cdots};(25,-6.5)*{\ }; 
"a4";"a2"**\crv{(15,6)&(5,7)}; 
\endxy\ ,$
and we put $\Lambda=\Lambda(s)\cap W_p(D_s)\cdot\lambda$ and $\Lambda'=\Lambda(s)\cap W_p(D_s)\cdot\lambda'$.
Let $\nu\in\Lambda$. Assume $\nu'\in{\rm Supp}(\nu)\cap W_p\cdot\lambda'$. Then we have seen that $\nu'\in\Lambda(s)$.
Moreover, the move $\nu\mapsto\nu'$ moves the arrow at the $a$-node to the $(a-1)$-node.
So the property $\nu'\in{\rm Supp}(\nu)\cap W_p\cdot\lambda'$ determines a map $\nu\mapsto\nu':\Lambda\to\Lambda(s)$ given by
\parbox[c][.9cm]{3.2cm}{$\begin{smallmatrix}
\cdots&{\rm o}&\land&\cdots&\mapsto&\cdots&\land&{\rm o}&\cdots\\
\cdots&{\rm o}&\vee&\cdots&\mapsto&\cdots&\vee&{\rm o}&\cdots\\
&&a&&&&&a&
\end{smallmatrix}
$}.
This map clearly preserves the order $\preceq$ and $W_p(D_s)$-conjugacy (under the dot action), so it has its image in $\Lambda'$.
Similarly, the property $\nu\in {\rm Supp}(\nu')\cap W_p\cdot\lambda$ determines a map $\nu'\mapsto\nu:\Lambda'\to\Lambda(s)$ given by reading the above rule in the opposite direction and this map preserves $\preceq$ and $W_p(D_s)$-conjugacy. So these maps are each others inverse and Proposition~\ref{prop.trans_equivalence} gives that $(T(\lambda):\Delta(\mu))=(T(\lambda'):\Delta(\mu'))$. Furthermore, since $\times$'s and empty nodes don't really play a role in the cap-curl diagram, it is obvious that $c_{\lambda'\mu'}$ is oriented if and only if $c_{\lambda\mu}$ is oriented.
When $\lambda=
\xy
(0,0)="a1";(0,0)*{\cdots};
(5,0)="a2";(5,1)*{\vee};
(10,0)="a3";(10,0)*{\cdots};
(15,0)="a4";(15,1)*{\vee};(15,-1)*{\land};
(20,0)="a5";(20,-1)*{\land};(20,-4)*{a};
(25,0)="a6";(25,0)*{\cdots};(25,-6)*{\ }; 
"a5";"a2"**\crv{(20,6)&(5,8)}; 
\endxy\ ,$
we choose
$\lambda'=
\xy
(0,0)="a1";(0,0)*{\cdots};
(5,0)="a2";(5,1)*{\vee};
(10,0)="a3";(10,0)*{\cdots};
(15,0)="a4";(15,-1)*{\land};
(20,0)="a5";(20,1)*{\vee};(20,-1)*{\land};(20,-4)*{a};
(25,0)="a6";(25,0)*{\cdots}; (25,-6)*{\ }; 
"a4";"a2"**\crv{(15,6)&(5,7)}; 
\endxy\ .$
We define $\Lambda$ and $\Lambda'$ as before and similar arguments as above give a bijection $\Lambda\to\Lambda'$ given by
\parbox[c][.9cm]{3.2cm}{$\begin{smallmatrix}
\cdots&\times&\land&\cdots&\mapsto&\cdots&\land&\times&\cdots\\
\cdots&\times&\vee&\cdots&\mapsto&\cdots&\vee&\times&\cdots\\
&&a&&&&&a&
\end{smallmatrix}
$}
with the same properties as before.
In this case we move a unique arrow from the $(a-1)$-node to the $a$-node to go from $\nu$ to $\nu'$, although we think of the move as the arrow at the $a$-node moving past the $\times$.
So in this case Proposition~\ref{prop.trans_equivalence} again gives that $(T(\lambda):\Delta(\mu))=(T(\lambda'):\Delta(\mu'))$. Furthermore, we again have that $c_{\lambda'\mu'}$ is oriented if and only if $c_{\lambda\mu}$ is oriented.

In case of a curl to the left of the wall, we move the left end node repeatedly one step to the left until it is the $0$-node. These moves are completely the same as the two types above (i.e. past an empty node or past an $\times$), only the inverse $\nu'\mapsto\nu$ of the final move is slightly different. The point is that a $\vee$ at the $0$-node could be replaced by a $\land$ without changing $\nu'$ and then be moved to the $1$-node. However, the condition that $\nu\in{\rm Supp}(\nu')$ should be $W_p(D_s)$-conjugate to $\lambda$ under the dot action singles out precisely one of the two options. So the property $\nu\in{\rm Supp}(\nu')\cap W_p(D_s)\cdot\lambda$ determines a map $\nu'\mapsto\nu:\Lambda'\to\Lambda(s)$ given by
\parbox[c][.9cm]{2.6cm}{$\begin{smallmatrix}
{\rm o}&\land&\cdots&\mapsto&\land&{\rm o}&\cdots\\
{\rm o}&\vee&\cdots&\mapsto&\vee&{\rm o}&\cdots
\end{smallmatrix}
$,}
where we assume that the parity of the number of $\land$'s in the arrow diagram of $\nu'$ is the same as that in the arrow diagram of $\lambda'$ (which has a $\land$ at the $0$-node).
This map preserves $\preceq$ and $W_p(D_s)$-conjugacy and has therefore image in $\Lambda$.
Finally, we replace in the arrow diagrams of $\lambda'$ and $\mu'$ the arrow at the $0$-node by its opposite. This turns the curl into a cap, doesn't change the orientedness of $c_{\lambda'\mu'}$, and $\lambda'$ and $\mu'$ stay the same. So after the final move and swap operation we again have by Proposition~\ref{prop.trans_equivalence} that $(T(\lambda):\Delta(\mu))=(T(\lambda'):\Delta(\mu'))$.

The case that the cap or curl is to the right of the wall is completely analogous. We now move the left end node of a cap to the right, and the right end node of a curl to the last node. 
There is no move where the type $D_s$ linkage principle is needed, and instead of the final swap to turn a curl into a cap we have a move around the right corner given by
\parbox[c][.9cm]{2.3cm}{$\begin{smallmatrix}
\cdots&\vee&\mapsto&\cdots&\land\\
\cdots&\land&\mapsto&\cdots&\vee
\end{smallmatrix}
$.}
In more detail, we define $\Lambda$ and $\Lambda'$ as before. Then $\nu\in\Lambda$ and $\nu'\in{\rm Supp}(\nu)\cap W_p\cdot\lambda'$ implies that  $\nu$ and $\nu'$ have the same number of arrows at each node, since this is true for $\lambda$ and $\lambda'$. So $\nu'$ is obtained from $\nu$ by changing the arrow at the last node to its opposite and this determines a map $\nu\mapsto\nu':\Lambda\to\Lambda(s)$ given by the above rule. 
Finally, it is easy to check that the above map preserves $\preceq$. For example, when $\nu$ is obtained from $\lambda$ by replacing an arrow pair $\vee\land$ resp. $\vee\vee$ to the right of the wall with the second arrow at the last node by the opposite arrow pair, then $\nu'$ is obtained from $\lambda'$ by replacing an arrow pair $\vee\vee$ resp. $\vee\land$ by the opposite arrow pair.
Clearly, this map preserves $W_p(D_s)$-conjugacy, so it has image in $\Lambda'$. Similarly, the property $\nu\in {\rm Supp}(\nu')\cap W_p\cdot\lambda$ determines a map $\nu'\mapsto\nu:\Lambda'\to\Lambda(s)$ given by the same rule and this map preserves $\preceq$ and $W_p(D_s)$-conjugacy and therefore has image in $\Lambda$.
So also for this move $c_{\lambda'\mu'}$ is oriented if and only if $c_{\lambda\mu}$ is oriented and Proposition~\ref{prop.trans_equivalence} gives that $(T(\lambda):\Delta(\mu))=(T(\lambda'):\Delta(\mu'))$.

Now we are reduced to the case of a cap for which the end nodes are consecutive. So $\lambda=
\xy
(0,0)="a1";(0,0)*{\cdots};
(5,0)="a2";(5,1)*{\vee};
(10,0)="a3";(10,-1)*{\land};(10,-4)*{a};
(15,0)="a4";(15,0)*{\cdots};
"a3";"a2"**\crv{(10,5)&(5,5)}; 
\endxy\ ,$
$a>0$, when the cap is to the left of the wall and $\lambda=
\xy
(0,0)="a1";(0,0)*{\cdots};
(5,0)="a2";(5,1)*{\vee};
(10,0)="a3";(10,-1)*{\land};(10,-4)*{\,a};
(15,0)="a4";(15,0)*{\cdots};(15,-7)*{\ }; 
"a3";"a2"**\crv{(10,-5)&(5,-5)}; 
\endxy\ ,$ when the cap is to the right of the wall. Now we choose
$\lambda'=
\xy
(0,0)="a1";(0,0)*{\cdots};
(5,0)="a3";(5,0)*{\bullet};
(10,0)="a3";(10,1)*{\vee};(10,-1)*{\land};(10,-4)*{a};
(15,0)="a4";(15,0)*{\cdots};(15,-7)*{\ }; 
\endxy\ .$
Define $\Lambda$ and $\Lambda'$ as before. First assume $a>1$ and let $\nu\in\Lambda$ and $\nu'\in {\rm Supp}(\nu)\cap W_p\cdot\lambda'$.
Then $\nu'\in\Lambda(s)$ as we have seen, and $\nu'$ is obtained from $\nu$ by moving the arrow at the $(a-1)$-node to the $a$-node.
Furthermore, this move can only be done when the arrows at the $(a-1)$-node and $a$-node are not both $\vee$ or both $\land$, i.e. when a cap connecting the two nodes is oriented.
Let us denote the set of $\nu\in\Lambda$ with this property by $\tilde\Lambda$. Then we obtain a map $\nu\mapsto\nu':\tilde\Lambda\to\Lambda(s)$ given by
\parbox[c][.9cm]{3.55cm}{$\begin{smallmatrix}
\cdots&\land&\vee&\cdots\\
\cdots&\vee&\land&\cdots
\end{smallmatrix}
\mapsto
\begin{smallmatrix}
\cdots&{\rm o}&\times&\cdots&\\
\end{smallmatrix}
$}
and it not hard to see that this map preserves $\preceq$ and $W_p(D_s)$-conjugacy and therefore has its image in $\Lambda'$.\footnote{For the preservation of $\preceq$ one can use functions like the $l_i(\lambda,\mu)$ in \cite[Sect~8]{CdV} and \cite[Sect~5]{BS}.} 

When $a=1$ we may have that $\nu'\in {\rm Supp}(\nu)\cap W_p\cdot\lambda'$ is not $W_p(D_s)$-conjugate to $\lambda'$. For example, we could have $\lambda=\vee\land\land\cdots$, $\lambda'={\rm o}\times\land\cdots$, $\nu=\land\land\vee\cdots=\vee\land\vee\cdots$, which all agree from the fourth node on. Then ${\rm Supp}(\nu)\cap W_p\cdot\lambda'=\{\nu'\}$, where $\nu'={\rm o}\times\vee\cdots$ is not $W_p(D_s)$-conjugate to $\lambda'$.
So in this case we define the map $\nu\mapsto\nu'$ by the property $\nu'\in {\rm Supp}(\nu)\cap W_p(D_s)\cdot\lambda'$, where $\nu\in\tilde\Lambda$, the set of all $\nu$ for which this intersection is nonempty. This map is then given by the same rule as above, where we assume that the parity of the number of $\land$'s in the arrow diagram of $\nu$ is the same as that in the arrow diagram of $\lambda$.

Now let $\nu'\in\Lambda'$ and $\nu\in{\rm Supp}(\nu')\cap W_p\cdot\lambda$. Then $\nu\in\Lambda(s)$ by the general fact at the start of the proof, and we see that $\nu=\nu^\pm\in\tilde\Lambda$, where $\nu^+$ resp. $\nu^-$ is obtained from $\nu'$ by moving the $\vee$ resp. $\land$ at the $a$-node to the $(a-1)$-node. So the above map has image equal to $\Lambda'$. Furthermore, it is easy to see that $\eta\preceq\nu$ implies $\eta^-\preceq\nu^-$ and $\eta^+\preceq\nu^+$. By Lemma~\ref{lem.JSF-arrows}(iii) we have that $\Hom_G(\nabla(\nu^+),\nabla(\nu^-))\ne0$.
Since $\lambda=\lambda^+$ we have by Proposition~\ref{prop.trans_projection} that $$(T(\lambda):\Delta(\mu))=(\widetilde T_{\lambda'}^\lambda T(\lambda'):\Delta(\mu))=(T(\lambda'):\widetilde T_\lambda^{\lambda'}\Delta(\mu))=(T(\lambda'):\Delta(\mu'))\,,$$ when $\mu=\mu^\pm$ for some $\mu'\in\Lambda'$, i.e. $\mu\in\tilde\Lambda$, and $0$ otherwise.
Here we used that for any finite dimensional $G$-module $M$ with a Weyl filtration $(M:\Delta(\mu))=\dim\Hom_G(M,\nabla(\mu))$.
Finally, $c_{\lambda\mu}$ is oriented if and only if our cap is oriented in $c_{\lambda\mu}$ and $c_{\lambda'\mu'}$ is oriented. So we can now finish by applying the induction hypothesis,
since $c_{\lambda'}$ has one cap or curl less than the original $c_\lambda$.
\end{proof}

Let $r$ be an integer $\ge1$. For any $\delta\in k$ one has the Brauer algebra $B_r(\delta)$; see e.g. \cite{Br}, \cite{Brown}, \cite{DorHanWal} or \cite{Wen} for a definition.
It has a family of {\it cell modules} $\mc S(\lambda)$, labelled by the partitions of $r,r-2,\ldots,$ up to $0$ or $1$. If $\lambda$ is $p$-regular and $\lambda\ne\emptyset$ when $\delta=0$ and $r$ even, then ${\mc S}(\lambda)$ has a simple head which we denote by ${\mc D}(\lambda)$. We will assume that $\delta$ is in the prime field, since otherwise $B_r(\delta)$ is isomorphic to a direct sum of matrix algebras over symmetric group algebras, see \cite[Prop~1.2]{DT}. We denote the transpose of a partition $\lambda$ by $\lambda^T$.

\begin{cornn}
Let $\lambda,\mu$ be partitions with $r-|\lambda|,r-|\mu|\ge0$ and even and assume that $\lambda\ne\emptyset$ if $r$ is even and $\delta=0$. Assume also that $\lambda_1+l(\lambda)\le p$. Choose $m\ge r$ such that $-2m=\delta$ mod $p$ and $s\in\{1,\ldots,\min(m,p)\}$ with $\lambda\in\Lambda(s)$.
Then
$$[\mc S(\mu^T):\mc D(\lambda^T)]=
\begin{cases}
1\text{\ if $\mu\preceq\lambda$ and $c_{\lambda\mu}$ is oriented,}\\
0\text{\ otherwise.}
\end{cases}
$$
\end{cornn}
\begin{proof}
Let $\Lambda$ be the set of partitions $\nu$ with $r-|\nu|$ even and $\ge0$. To any such $\nu$ we associate $t=|\nu|$ and write $r-t=2u$. Denote the Specht module for the symmetric group $\Sym_t$ associated to $\nu$ by $S(\nu)$, and, for $\nu$ $p$-regular, denote the irreducible by $D(\nu)$. Define the twisted cell and irreducible modules $\widetilde{\mc S}(\nu)$ and $\widetilde{\mc D}(\nu)$ as in \cite[Sect~1.2]{DT}. Note that in \cite{DT} the transpose of a partition $\nu$ is denoted by $\nu'$. Since $k_{\rm sg}\ot S(\nu)=S(\nu^T)^*$ and $S(\nu^T)$ have the same composition factors with multiplicities, and the functor $Z_u\ot_{k\Sym_t}-$ is exact,
we must have that $\widetilde{\mc S}(\nu)$ and ${\mc S}(\nu^T)$ have the same composition factors with multiplicities.
Furthermore, $\widetilde{\mc D}(\nu)={\mc D}(\varphi(\nu))$, for $\nu\in\Lambda$ $p$-regular and $\ne\emptyset$ when $\delta=0$ and $r$ even, where $\varphi$ is the Mullineux involution defined by $k_{\rm sg}\ot D(\lambda)=D(\varphi(\lambda))$.
It follows that $[\widetilde{\mc S}(\nu):\widetilde{\mc D}(\eta)]=[{\mc S}(\nu^T):{\mc D}(\varphi(\eta))]$ for $\nu,\eta\in\Lambda$ with $\eta$ $p$-regular and $\ne\emptyset$ when $\delta=0$ and $r$ even. Now $l(\lambda)+\lambda_1\le p$ implies that $\lambda$ is a $p$-core, which, in turn, implies that $\lambda$ and $\lambda^T$ are $p$-regular. 
So $S(\lambda)$ is irreducible and $\varphi(\lambda)=\lambda^T$.

Using the symplectic Schur functor we have, by \cite[Prop~2.1(iii)]{DT} and the above,
$(T(\lambda):\Delta(\mu))=[\widetilde{\mc S}(\mu):\widetilde{\mc D}(\lambda)]=[{\mc S}(\mu^T):{\mc D}(\lambda^T)]$. So the result follows from the previous theorem.
\end{proof}

\section{Decomposition numbers}\label{s.dec_num}
Let $\mu\in\Lambda_m$. Choose $s\in\{1,\ldots,\min(m,p)\}$ with $\mu\in\Lambda(s)$.
First we define the \emph{cap-curl codiagram} $co_\mu$ of the arrow diagram associated to $\mu\in X^+$ as follows. All caps and curls are clockwise, starting from the arrow closest to the wall.
We start on the left side of the wall. We first form the caps recursively. Find an arrow pair $\land\vee$ that are neighbours in the sense that the only arrows in between are already connected with a cap or are part of an $\times$, and connect them with a cap. Repeat this until there are no more such arrow pairs.
Now the unconnected arrows that are not part of an $\times$ form a sequence $\vee\cdots\vee\land\cdots\land$. We connect consecutive (in the mentioned sequence) $\vee\vee$ pairs with a curl, starting from the left. At the end the unconnected arrows that are not part of an $\times$ form a sequence $\vee\land\cdots\land$ or just a sequence of $\land$'s.
The caps on the right side of the wall are formed in the same way. The curls now connect consecutive $\land\land$ pairs and are formed starting from the right. So at the end the unconnected arrows that are not part of an $\times$ form a sequence $\vee\cdots\vee\land$ or just a sequence of $\vee$'s. Note that none of these arrows occur inside a cap or curl. For example, when $p=23$, $m=17$, $s=12$ and $\mu=(87^763^21)$, then
$co_\mu$ is
$$
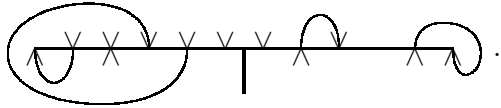
\xy
(0,0)="a1";(0,-1)*{\land};
(5,0)="a2";(5,1)*{\vee};
(10,0)="a3";(10,1)*{\vee};(10,-1)*{\land};
(15,0)="a4";(15,1)*{\vee};
(20,0)="a5";(20,1)*{\vee};
(25,0)="a6";(25,1)*{\vee};(27.5,-3)*{\rule[0mm]{.3mm}{6mm}};
(30,0)="a7";(30,1)*{\vee};
(35,0)="a8";(35,-1)*{\land};
(40,0)="a9";(40,1)*{\vee};
(45,0)="a10";
(50,0)="a11";(50,-1)*{\land};
(55,0)="a12";(55,-1)*{\land};
{"a1";"a12"**@{-}}; 
"a2";"a1"**\crv{(5,-6)&(0,-6)}; 
"a9";"a8"**\crv{(40,6)&(35,6)}; 
"a5";"a4"**\crv{(20,-8)&(-5,-9)&(-5,7)&(15,7)}; 
"a11";"a12"**\crv{(50,4)&(59,4)&(59,-4)&(55,-4)}; 
\endxy\ .$$

Let $\lambda\in\Lambda_m$ with $\mu\preceq\lambda$. If necessary, we change $s$ (and the arrow diagram of $\mu$, and $co_\mu$) to make sure that $\lambda\in\Lambda(s)$. Then the arrow diagram of $\lambda$ has its single arrows and its $\times$'s at the same nodes as the arrow diagram of $\mu$. If the arrow diagram of $\mu$ has an arrow at $0$, then we assume that the parity of the number of $\land$'s in the arrow diagram of $\lambda$ is the same as that for $\mu$. This only requires a possible change of an arrow at $0$ to its opposite in the arrow diagram of $\lambda$. If there is no arrow at $0$, then these parities will automatically be the same, since $\lambda$ is $W_p(D_{l(\lambda)})$-conjugate to $\mu$ under the dot action.
Then we know, by Remark~\ref{rems.preceq}.1, that the arrow diagram of $\lambda$ can be obtained from that of $\mu$ by repeatedly replacing an arrow pair $\land\vee$ to the left or to the right of the wall, or an arrow pair $\vee\vee$ to the left of the wall, or an arrow pair $\land\land$ to the right of the wall, by the opposite arrow pair.
Now we define the \emph{cap-curl codiagram} $co_{\mu\lambda}$ associated to $\mu$ \emph{and $\lambda$} by replacing each arrow in $co_\mu$ by the arrow from the arrow diagram of $\lambda$ at the same node. Put differently, we put the caps and curls from $co_\mu$ on top of the arrow diagram of $\lambda$. We say that $co_{\mu\lambda}$ is \emph{oriented} if all caps and curls in $co_{\mu\lambda}$ are oriented (clockwise or anti-clockwise). It is not hard to show that when $co_{\mu\lambda}$ is oriented, the arrow diagrams of $\mu$ and $\lambda$ are the same at the nodes which are not endpoints of a cap or a curl in $co_\mu$.

For example, when $p=11$, $m=7$, $s=5$ and $\mu=(4^321)$. Then $\rho_s=3$ and $co_\mu$ is
$$
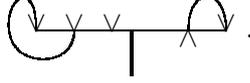
\xy
(0,0)="a1";(0,1)*{\vee};
(5,0)="a2";(5,1)*{\vee};
(10,0)="a3";(10,1)*{\vee};(12.5,-3)*{\rule[0mm]{.3mm}{6mm}};
(15,0)="a4";
(20,0)="a5";(20,-1)*{\land};
(25,0)="a6";(25,1)*{\vee};
{"a1";"a6"**@{-}}; 
"a2";"a1"**\crv{(5,-4)&(-4,-5)&(-4,5)&(0,5)}; 
"a6";"a5"**\crv{(25,6)&(20,6)} 
\endxy\ .$$
Consider two dominant weights $\lambda$ with $\mu\preceq\lambda$: $(6^332)$ and $(5^2432)$
with arrow diagrams $\vee\land\land{\rm o}\vee\land$ and $\land\land\vee{\rm o}\vee\land$.
Only for the last $co_{\mu\lambda}$ is oriented. Note that for the first we are not allowed to change the $\vee$ at the $0$-node to a $\land$,
because then the parity of the number of $\land$'s in the arrow diagram of $\lambda$ would not be the same as that in the arrow diagram of $\mu$.

\begin{thm}\label{thm.dec_num}
Let $s\in\{1,\ldots,\min(m,p)\}$, $\lambda\in\Lambda(s)$ and $\mu\in X^+$. Then
$$[\nabla(\lambda):L(\mu)]=[\Delta(\lambda):L(\mu)]=
\begin{cases}
1\text{\ if $\mu\preceq\lambda$ and $co_{\mu\lambda}$ is oriented,}\\
0\text{\ otherwise.}
\end{cases}
$$
\end{thm}
\begin{proof}
The proof is by induction on the number of caps and curls in $co_\mu$ and is completely analogous to the proof of Theorem~\ref{thm.filt_mult}.
The role of $\lambda$ is now played by $\mu$. We leave the details to the reader. The final argument involving the projection is as follows.
By Lemma~\ref{lem.JSF-arrows}(iii) we have that $\Hom_G(\nabla(\nu^+),\nabla(\nu^-))\ne0$.
Since $\mu=\mu^-$ we have by Proposition~\ref{prop.trans_projection} that
\begin{align*}[\Delta(\lambda):L(\mu)]&=(I_\Lambda(\mu):\nabla(\lambda))=(\widetilde T_{\mu'}^\mu I_{\Lambda'}(\mu'):\nabla(\lambda))=(I_{\Lambda'}(\mu'):\widetilde T_\mu^{\mu'}\nabla(\lambda))\\
&=(I_{\Lambda'}(\mu'):\nabla(\lambda'))=[\Delta(\lambda'):L(\mu')]\,,
\end{align*}
when $\lambda=\lambda^\pm$ for some $\lambda'\in\Lambda'$, i.e. $\lambda\in\tilde\Lambda$, and $0$ otherwise.
\end{proof}

Define the involution $\dagger$ on $\Lambda(s)$ by letting $\lambda^\dagger$ be the partition whose arrow diagram is obtained from that of $\lambda$ by replacing all single arrows by their opposite. Note that $\dagger$ reverses the order $\preceq$.
\begin{cornn}
Let $\lambda,\mu\in\Lambda(s)$. Then $[\Delta(\lambda):L(\mu)]=(T(\mu^\dagger):\nabla(\lambda^\dagger))$.
\end{cornn}
\begin{proof}
This follows from Theorems~\ref{thm.filt_mult} and \ref{thm.dec_num}, since $co_{\mu\lambda}$ is obtained form $c_{\mu^\dagger\lambda^\dagger}$ by replacing all single arrows by their opposite and reflecting all caps and curls in the horizontal axis.
\end{proof}
\begin{rem}
In view of \cite[Lem A4.6]{D} and the above corollary it is natural to conjecture that, for $\Lambda$ the intersection of $\Lambda(s)$ with a $W_p(D_s)$-orbit under the dot action, the algebra $(O_{\Lambda^\dagger}(k[G])^*,\preceq)$ is the Ringel dual of $(O_\Lambda(k[G])^*,\preceq)$.
\end{rem}

\noindent {\bf Statements and Declarations}. The authors declare that they have no conflict of interest.

\end{document}